\documentclass{article}
\usepackage{metalogo}

\usepackage{verbatim}

\usepackage{amsmath}
\usepackage{amsmath,amscd}
\usepackage{amsthm}
\usepackage{amsfonts}
\usepackage{amssymb}
\usepackage{mathrsfs}
\usepackage{fancyhdr}
\usepackage[all]{xy}

\usepackage{amssymb, paralist, xspace, graphicx, url, amscd, euscript, mathrsfs, stmaryrd}
\usepackage[all]{xy}

\numberwithin{equation}{section}
2
\numberwithin{subsection}{section}
\allowdisplaybreaks[1]

\newenvironment{enumerate1} {\begin{enumerate}[\upshape (1)]} {\end{enumerate}}

\theoremstyle{definition}
\newtheorem{definition}[subsubsection]{Definition}
\newtheorem{theorem}[subsubsection]{Theorem}
\newtheorem{lemma}[subsubsection]{Lemma}

\newtheorem{corollary}[subsubsection]{Corollary}

\newtheorem{Proposition}[subsubsection]{Proposition}

\newtheorem{Claim}[]{Claim}
\newtheorem*{Organization}{Organization}
\newtheorem{construction}[subsubsection]{Construction}

\theoremstyle{remark}

\newtheorem{remark}{Remark}

\newcommand{\p}{\mathbf{P}^{1}}

\newcommand{\blA}{Bl_{Z}A}
\newcommand{\HS}{\textrm{Hilb}_{S}}
\newcommand{\HC}{\textrm{Hilb}_{C}}

\newcommand{\Fl}{\textrm{Flag}_{S}}

\newcommand{\RH}{\mathcal{R}Hom}

\newcommand{\J}{\overline{J}}
\newcommand{\hig}{\mathcal{H}iggs}

\newcommand{\de}{\textrm{det}}
\newcommand{\Pd}{\textrm{Pic}^{d}}
\newcommand{\Hom}{\mathcal{H}om}
\newcommand{\FB}{\textrm{Flag}^{d}_{\mathcal{B}}}

\begin{document}
\title{Construction of the Poincar\'e sheaf on the stack of rank $2$ Higgs bundles}

\author{Mao Li}

\maketitle

\begin{abstract}
Let $X$ be a smooth projective curve of genus $g$ and $L$ be a degree $l$ line bundle on $X$ with $l\geq 2g-1$. Denote the stack of rank two Higgs bundles on $X$ with value in $L$ by $\hig$ and the semistable part by $\hig_{ss}$. Let $H$ be the Hitchin base. In this paper we will construct the Poincar\'e sheaf $\mathcal{P}$ on $\hig\times_{H}\hig_{ss}$ which is a maximal Cohen-Macaulay sheaf and flat over $\hig_{ss}$. In particular this includes the locus of nonreduced spectral curves. The present work generalizes the construction of the Poincar\'e sheaf in $[1]$, $[3]$ and $[13]$. 
\end{abstract}

\tableofcontents

\newpage

\section{Introduction}

\subsection{Poincar\'e sheaf}
Let $C$ be a smooth projective curve and $J$ be the Jacobian of $C$. Then it is well known that there is a Poincar\'e line bundle $\mathcal{P}$ on $J\times J$ which is the universal family of topologically trivial line bundles on $J$(See $[14]$).
When $C$ is an integral planar curve, the Jacobian $J$ is no longer projective, but we can consider the compactified Jacobian $\J$ ($[7]$,$[8]$) which parameterizes torsion free rank $1$ sheaves on $C$. In this case there exists a Poincar\'e line bundle on $\mathcal{P}$ on $J\times\J$ ($[2]$) defined in the following way. Consider
\[\begin{CD}
C\times J\times\J\\
@V{\pi}VV\\
J\times\J\\
\end{CD}\]
Then:
\begin{equation}\label{Poincare line bundle}
\mathcal{P}=\de(R\pi_{*}(L\otimes F))\otimes \de(R\pi_{*}O_{C})\otimes \de(R\pi_{*}(L))^{-1}\otimes \de(R\pi_{*}(F))^{-1}
\end{equation}
where $F$ and $L$ are the universal sheaves on $C\times J$ and $C\times\J$. It is interesting to ask whether we can extend $\mathcal{P}$ to $\J\times\J$. For curves with double singularities, this has been answered in $[15]$, and the generalization to all integral planar curves is obtained in $[1]$(Similar results have also been obtained by Margarida Melo, Antonio Rapagnetta and Filippo Viviani in $[3]$, where they work with moduli space instead of stacks):
\begin{theorem}
There exists a maximal Cohen-Macaulay sheaf $\mathcal{P}$ on $\J\times\J$ such that the restriction of $\mathcal{P}$ to $J\times\J\cup\J\times J$ is the Poincar\'e line bundle given by formula~\ref{Poincare line bundle}. Moreover, $\mathcal{P}$ is flat over both component $\J$.
\end{theorem}
We shall call $\mathcal{P}$ the Poincar\'e sheaf. In fact, even though the theorem is stated only for integral curves, the argument presented in $[1]$ also works for reduced planar curves.\\

\begin{remark}
For the construction of the Poincar\'e line bundle on $J\times\J\cup\J\times J$, we do not need to assume that $C$ is reduced. Similarly, Lemma~\ref{equivariance property} also holds for Poincar\'e line bundles of nonreduced planar curves.
\end{remark}

One of the main motivations for studying compactified Jacobians is that they are fibers of the Hitchin fibration. 
Let $X$ be a smooth projective curve, $L$ a line bundle on $X$. Denote the stack of rank $n$ Higgs bundle with value in $L$ by $\hig$. Let $H$ be the Hitchin base which parameterizes spectral curves. We have the Hitchin fibration $\hig\xrightarrow{h} H$. It is well-known that a Higgs bundle on $X$ can be naturally viewed as a torsion-free rank $1$ sheaf on the spectral curve $C$(See $[16]$). Moreover, let $H_{r}$ be the open subscheme of $H$ corresponding to reduced spectral curves. Then it is well-known that the fibers of $h$ over $H_{r}$ can be identified with the compactified Jacobian of $C$. Let $\hig^{reg}$ denote the open substack of $\hig$ corresponding to line bundles on spectral curves. Then formula~\ref{Poincare line bundle} defines a Poincar\'e line bundle $\mathcal{P}$ on $$\hig^{reg}\times_{H}\hig\cup\hig\times_{H}\hig^{reg}$$
Moreover, denote the open substack of Higgs bundles with generically regular Higgs field by $\widetilde{\hig}$ (Definition~\ref{generically regular Higgs field}). As we will see later in Proposition~\ref{definition of Q}, the construction in $[1]$ actually provides a maximal Cohen-Macaulay sheaf $\mathcal{P}$ on
$$\widetilde{\hig}\times_{H}\hig\cup\hig\times_{H}\widetilde{\hig}$$
which extends the Poincar\'e line bundle on
$$\hig^{reg}\times_{H}\hig\cup\hig\times_{H}\hig^{reg}$$
Moreover, it is shown in $[1]$ that $\mathcal{P}$ induces an autoequivalence of the derived category. This establishes the Langlands duality for Hitchin systems for $GL(n)$ over the locus of integral spectral curves(See $[1]$ for discussions about its relations with automorphic sheaves). Hence it is a very interesting question whether we can extend the maximal Cohen-Macaulay sheaf above to $\hig\times_{H}\hig$. The main issue here is how to extend $\mathcal{P}$ to the locus of nonreduced spectral curves. In this paper we provide a partial answer in the case of rank $2$ Higgs bundles on a projective smooth curve $X$. Namely, let $\hig_{ss}$ be the open substack of semistable Higgs bundles. We are going to construct a maximal Cohen-Macaulay sheaf on $\hig\times_{H}\hig_{ss}$ such that it is an extension of the Poincar\'e line bundle.

\subsection{Main result and the Organization of the paper}
Let $X$ be a smooth projective curve of genus $g$, $L$ a line bundle on $X$ of degree $l\geq 2g-1$. Let $\hig$ be the stack of rank two Higgs bundles on $X$ with value in $L$, and $H$ be the Hitchin base. Let $\hig_{ss}$ denote the open substack of semistable Higgs bundles, and let $\hig^{reg}$ denote the open substack of $\hig$ corresponds to Higgs bundles that are line bundles on its spectral curve. The main theorem of the paper is the following:
\begin{theorem}\label{the main theorem}
There exists a maximal Cohen-Macaulay sheaf $\mathcal{P}$ on $\hig\times_{H}\hig_{ss}$ which extends the Poincar\'e line bundle on $\hig^{reg}\times_{H}\hig_{ss}$. Moreover, $\mathcal{P}$ is flat over $\hig_{ss}$.
\end{theorem}
By Proposition~\ref{uniqueness of CM sheaf}, such an extension of the Poincar\'e line bundle on $\hig^{reg}\times_{H}\hig_{ss}$ is unique if exists. Let $\widetilde{\hig}$ be the open substack of $\hig$ corresponds to Higgs bundles with generically regular Higgs field(See Definition~\ref{generically regular Higgs field}). The construction in $[1]$ provides the Poincar\'e sheaf on $\widetilde{\hig}\times_{H}\hig_{ss}$. The problem is that for Higgs bundles with nonreduced spectral curves, the Higgs field may not be generically regular, hence the construction in $[1]$ does not extend to $\hig\times_{H}\hig_{ss}$. In the previous work $[13]$ we constructed the Poincar\'e sheaf for the stack of Higgs bundles of $\p$. In this paper we generalize the construction to higher genus curves.
\begin{Organization}
The rest of the paper is organized as follows. In $1.3$ we shall review the construction in $[1]$ and adapt it to our setting. In $1.4$ we sketch the main idea of the construction, which is the fundamental diagram $1.2$ below. We will discuss its motivation and also the relation with our previous work for $\p$. In section $2$ we gather some preliminary results that will be used in the paper. Most of these are similar to the previous paper, the new feature is the vanishing result in Lemma~\ref{vanishing of pushforward}. In section $3$ we first review the geometry of the stack of Higgs bundles. In subsection $3.3$ we come to the main construction, which involves the stack $\mathcal{Y}$, a closed substack $\mathcal{Z}$ of $\mathcal{Y}$ and the blowup of $\mathcal{Y}$ along $\mathcal{Z}$ which is denoted by $\mathcal{B}$. The main result is the fundamental diagram established in Proposition~\ref{resolve the rational map}, which roughly speaking says the following(See the discussions in subsection $1.4$ and subsection $3.3$ for more detail): There exists a natural morphism
$$\mathcal{Y}\backslash\mathcal{Z}\rightarrow \HS^{d}$$
and it extends to a morphism $\mathcal{B}\rightarrow \HS^{d}$, hence we have the following diagram:
\begin{equation}\label{first appearence of the diagram}
\xymatrix{
 & \mathcal{B} \ar[d]_{\pi} \ar[r]^{p} & \HS^{d}\\
\mathcal{Z} \ar[r] & \mathcal{Y} & \\
}
\end{equation}
In section $4$ we use the fundamental diagram above to establish Lemma~\ref{the main lemma}, and in the last section we use it to finish the construction of the Poincar\'e sheaf.
\end{Organization}

\subsection{Review of the construction of the Poincar\'e sheaf}
This subsection is essentially the same as subsection $1.3$ of the previous work $[13]$, we include it here for reader's convenience. We will review the construction of the Poincar\'e sheaf in $[1]$ and $[3]$. And we will also adapt the construction to our setup. Let $C$ be a reduced planar curve embedded into a smooth surface $C\hookrightarrow S$. It is well known that $\HC^{d}$ is a complete intersection in $\HS^{d}$ of codimension $d$. Let $\overline{J}$ be the stack of torsion free rank $1$ sheaves on $C$. The Poincar\'e sheaf on $\J\times\overline{J}$ can be constructed as follows. Fix an ample line bundle $N$ of $C$. We have a natural Abel-Jacobian map:
$$\HC^{d}\xrightarrow{\alpha}  \overline{J} \qquad D\rightarrow I_{D}^{\vee}\otimes N^{-m}$$
Let $U_{d}$ be the open subscheme of $\HC^{d}$ given by the condition $H^{1}(I_{d}^{\vee})=0$. Then the restriction of $\alpha$ to $U_{d}$ is smooth, and the union of the image of all $U_{d}$ in $\overline{J}$ covers $\overline{J}$. So we only need to construct Poincar\'e sheaf on $U_{d}\times \overline{J}$ and show it descends to $\overline{J}\times \overline{J}$. Let $F$ be the universal sheaf on $C\times \overline{J}$. The Hilbert scheme of the surface is denoted by $\HS^{d}$. It is well known that $\HS^{d}$ is smooth. Let $\Fl^{d}$ be the flag Hilbert scheme of $S$, which parameterizes length $d$ subschemes together with a complete flag:
$$\emptyset=D_{0}\subseteq D_{1}\subseteq\cdots\subseteq D_{d}=D$$
Consider the following diagram:
\[\begin{CD}
\HS\times \overline{J}@<{\psi\times id}<<\widetilde{\HS}\times \overline{J}@>{\sigma_{d}\times id}>>S^{d}\times \overline{J}@<{l^{d}\times id}<<C^{d}\times \overline{J}\\
@V{p_{1}}VV\\
\HS
\end{CD}\]
where $\widetilde{\HS}$ stands for the isospectral Hilbert scheme of $S$ (See $[1]$ Proposition $3.7$ or $[17]$ for the definition). It is known that $\psi$ is finite flat. Moreover, let $\HS^{' d}$ be the open subscheme of $\HS^{d}$ parameterizing subschemes that can be embedded into smooth curves, Then we have:
$$\widetilde{\HS^{d}}|_{\HS^{' d}}\simeq\Fl^{' d}=\Fl^{d}|_{\HS^{' d}}$$
Let $\mathcal{D}\xrightarrow{\pi} \HS$ be the universal finite subscheme over $\HS$ and set $\mathcal{A}=\pi_{*}O_{D}$, then we define:
\begin{equation}
Q=((\psi\times id)_{*}(\sigma_{d}\times id)^{*}(l^{d}\times id)_{*}F^{\boxtimes d})^{sign}\otimes p_{1}^{*}det(\mathcal{A})^{-1}
\end{equation}
where the upper index "sign" stands for the space of antiinvariants with respect to the natural action of the symmetric group. Then it is proved in $[1]$ that $Q$ is supported on $\HC$ and it's a maximal Cohen-Macaulay sheaf. Moreover, if we restrict $Q$ to $U_{d}$, then it descends down to $\J\times \J$. (In $[1]$ the statement is proved only for integral curves, but the same argument works for any reduced planar curves. The construction also works for families of planar curves). Let $\HS^{'}\subseteq \HS$ be the open subscheme parameterizing subchemes $D$ such that $D$ can be embedded into a smooth curve. Then we have a simpler description of the restriction of $\widetilde{\HS}$ to $\HS^{'}$ thanks to the following proposition ($[1]$ Proposition $3.7$):
\begin{Proposition}\label{rewrite using flag}
Let $\Fl^{d}$ be the flag Hilbert scheme of $S$, then we have
$$\widetilde{\HS^{d}}|_{\HS^{' d}}\simeq \Fl^{' d}=\Fl^{d}|_{\HS^{' d}}$$
\end{Proposition}
Hence over $\HS^{'}$ we have $\Fl^{'}\simeq \widetilde{\HS}$, and the construction can be written in terms of $\Fl^{'}$.\\

We shall adapt the construction above to our setting. Namely, let $\hig$ be the stack of rank $2$ Higgs bundles on $X$ with value in $L$, and $x_{0}$ a point on $X$. We work with the family of spectral curves over $H$:
$$\xymatrix{
\mathcal{C} \ar[rr] \ar[rd] & & S\times H \ar[ld]\\
& H\\
}
$$
Let $N$ be the line bundle on $\mathcal{C}$ which is the pullback of $O(x_{0})$ on $X$. Let $\widetilde{\hig}$ be the stack of Higgs bundles with generically regular Higgs field. Then we still have Abel-Jacobian map:
$$\textrm{Hilb}^{d}_{\mathcal{C}|H}\xrightarrow{\alpha}\widetilde{\hig}:\qquad D\rightarrow I_{D}^{\vee}\otimes N^{-m}$$
Moreover, if we set $U_{d}$ to be the open subscheme of $\textrm{Hilb}^{d}_{\mathcal{C}|H}$ consists of $D$ such that $H^{1}(\check{I}_{D})=0$, then $\alpha$ is smooth on $U_{d}$. Also if we vary $d$, the image of all $U_{d}$ will cover $\widetilde{\hig}$.\\
Consider the diagram:
\[\begin{CD}
\HS\times\hig@<{\psi\times id}<<\widetilde{\HS}\times\hig@>{\sigma_{d}\times id}>>S^{d}\times\hig@<{l^{d}\times id}<<\mathcal{C}^{d}\times_{H}\hig\\
@V{p_{1}}VV\\
\HS
\end{CD}\]
where $\mathcal{C}^{d}$ is $d$ fold Cartesian product of $\mathcal{C}$ over $H$:
$$\mathcal{C}^{d}=\mathcal{C}\times_{H}\cdots\times_{H}\mathcal{C}$$
Similarly, set:
\begin{equation}\label{Q}
Q=((\psi\times id)_{*}(\sigma_{d}\times id)^{*}(l^{d}\times id)_{*}F^{\boxtimes d})^{sign}\otimes p_{1}^{*}det(\mathcal{A})^{-1}
\end{equation}
Then by essentially the same argument, we get a Cohen-Macaulay sheaf $Q$ of codimension $d$ on $\HS^{d}\times\hig$. If we denote the open subscheme of $H$ corresponding to reduced spectral curves by $H_{r}$, and its complement by $H_{nr}$. Then over $H_{r}$, the sheaf $Q$ is supported on $\textrm{Hilb}^{d}_{\mathcal{C}|H_{r}}\times_{H_{r}}\hig|_{H_{r}}$. It is not hard to check that $H_{nr}$ has codimension $2l+1-g$ in $H$. Also since $\hig$ is flat over $H$, the complement of $\HS^{d}\times\hig|_{H_{r}}$ also has codimension $2l+1-g$. From the construction in $[1]$ it is not hard to check that the codimension of
$$\textrm{Supp}(Q)\cap\HS^{d}\times\hig|_{H_{nr}}$$
is greater than or equals to $d+2l+1-g$. Since $Q$ is Cohen-Macaulay of codimension $d$, the support of $Q$ is of pure dimension without embedded components and has codimension equal to $d$ in $\HS^{d}\times\hig$ (Proposition~\ref{codimension of supp of CM sheaves}). We conclude that $\textrm{Hilb}^{d}_{\mathcal{C}|H_{r}}\times_{H_{r}}\hig|_{H_{r}}$ is a dense open subset in $\textrm{Supp}(Q)$ with the codimension of the complement greater than or equals to $2l+1-g$. Hence $Q$ is supported on $\textrm{Hilb}^{d}_{\mathcal{C}|H}\times_{H}\hig$ over the entire $H$. In conclusion, we have:
\begin{Proposition}\label{definition of Q}
Let $Q$ be the sheaf on $\HS^{d}\times\hig$ given by formula~\ref{Q}, then
\begin{enumerate1}
\item $Q$ is a maximal Cohen-Macaulay sheaf of codimension $d$ on $\HS^{d}\times\hig$ supported on $\textrm{Hilb}^{d}_{\mathcal{C}|H}\times_{H}\hig$.
\item If we consider its restriction to $U_{d}\times_{H}\hig$ (Recall that $U_{d}$ is the open subscheme of $\textrm{Hilb}^{d}_{\mathcal{C}|H}$ consists of $D$ such that $H^{1}(\check{I}_{D})=0$), then it descend to $\widetilde{\hig}\times_{H}\hig$ and agrees with the Poincar\'e line bundle on $\hig^{reg}\times_{H}\hig$.
\end{enumerate1}
\end{Proposition}

Using the fact that the complement of $\hig^{reg}$ has codimension greater than or equals to $2$, we conclude from Proposition~\ref{uniqueness of CM sheaf} that the maximal Cohen-Macaulay sheaves on $\widetilde{\hig}\times_{H}\hig$ and $\hig\times_{H}\widetilde{\hig}$ constructed from the previous Proposition glues together, hence we have the following:
\begin{corollary}\label{CM sheaf constructed in 1}
We have a maximal Cohen-Macaulay sheaf on
$$\widetilde{\hig}\times_{H}\hig\cup\hig\times_{H}\widetilde{\hig}$$ which agrees with the Poincar\'e line bundle on $$\hig^{reg}\times_{H}\hig\cup\hig\times_{H}\hig^{reg}$$
\end{corollary}

For later use, we will also denote
\begin{equation}\label{Q'}
Q'=((\sigma_{d}\times id)^{*}(l^{d}\times id)_{*}F^{\boxtimes d})\otimes (id\times\psi)^{*}(p_{1}^{*}det(\mathcal{A})^{-1})
\end{equation}
The following lemma is clear from the formula~\ref{Q}:
\begin{lemma}\label{summand}
$Q$ is a direct summand of $(\psi\times id)_{*}(Q')$
\end{lemma}

For the purpose of the last section, let us also record the following equivariance properties of the Poincar\'e sheaf established in $[1]$. All the properties still hold for the Poincar\'e sheaf in Corollary~\ref{CM sheaf constructed in 1}, because of Proposition~\ref{uniqueness of CM sheaf}.
\begin{lemma}\label{equivariance property}
\begin{enumerate1}
\item Let $L$ be a line bundle on $C$. Consider the automorphism $\mu_{L}$ of $\J$ defined by $F\rightarrow F\otimes L$, then we have that:
$$(\mu_{L}\times id)^{*}\mathcal{P}\simeq \mathcal{P}\otimes p_{2}^{*}(\mathcal{P}_{L})$$
where $\mathcal{P}_{L}$ is a line bundle on $\J$ obtained by restriction of the Poincar\'e line bundle to $\{L\}\times\J\hookrightarrow J\times\J$
\item Let $\nu$ be the involution of $\J$ given by $F\rightarrow \check{F}=\Hom_{O_{C}}(F,O_{C})$. Consider:
$$\J\times\J\xrightarrow{\nu\times\textrm{id}}\J\times\J$$
Then we have:
$$(\nu\times\textrm{id})^{*}(\mathcal{P})\simeq\check{\mathcal{P}}=\Hom(\mathcal{P},O)$$
\item Consider the diagram:
$$\xymatrix{
J\times\J & J\times\J\times\J \ar[l]_{p_{13}} \ar[r]^{p_{23}} \ar[d]^{\mu\times\textrm{id}} & \J\times\J\\
 & \J\times\J
}
$$
where $J\times\J\xrightarrow{\mu}\J$ is given by $(L,F)\rightarrow L\otimes F$. Then we have
$$(\mu\times\textrm{id})^{*}(\mathcal{P})\simeq p_{13}^{*}(\mathcal{P})\otimes p_{23}^{*}(\mathcal{P})$$
\end{enumerate1}
\end{lemma}

\subsection{Motivation of the construction and a reformulation of the main theorem}
In this subsection we will sketch the idea of the construction of the Poincar\'e sheaf. The motivation comes from Drinfeld's construction of the automorphic sheaf which we shall briefly recall.(See $[11]$ and $[12]$)
Let $X$ be a smooth projective curve over a finite field, and $\mathcal{E}$ an irreducible rank two $l$-adic local system on $X$. Denote the stack of rank two vector bundles of $X$ by $\textrm{Bun}_{2}$. In $[11]$ Drinfeld constructed an automorphic perverse sheaf $\textrm{Aut}_{\mathcal{E}}$ on $\textrm{Bun}_{2}$ via the following procedure. Let $d$ be an integer greater than $4g-4$. Let $S=\Pd$ be the Picard scheme of $X$ corresponding to degree $d$ line bundles. Let $P=X^{(d)}$. It is well known that the symmetric power $X^{(d)}$ is a projective bundle over $\Pd$. Let $\check{P}$ be the dual projective bundle. By definition, $\check{P}$ classifies nontrivial extensions up to scalers:
$$0\rightarrow \Omega^{1}_{X}\rightarrow L_{2}\rightarrow L_{1}\rightarrow 0$$
$\check{P}$ is equipped with natural morphisms:
$$\check{P}\rightarrow \textrm{Bun}_{2}:  \{0\rightarrow \Omega^{1}_{X}\rightarrow L_{2}\rightarrow L_{1}\rightarrow 0\} \rightarrow L_{2}$$
$$\check{P}\rightarrow S: \{0\rightarrow \Omega^{1}_{X}\rightarrow L_{2}\rightarrow L_{1}\rightarrow 0\} \rightarrow L_{1}$$
There exists a nonempty open subscheme $U\subseteq \check{P}$ such that the morphism $U\rightarrow \textrm{Bun}_{2}$ is smooth. Let $Z$ be the scheme of universal hyperplane in $P\times_{S}\check{P}$. Hence we have the following commutative diagram:
$$\xymatrix{
 & Z \ar[ld]_{\check{\rho}} \ar[rd]^{\rho} & \\
\check{P} \ar[rd]_{\check{\pi}} \ar[d]_{\sigma} & & P \ar[ld]^{\pi}\\
\textrm{Bun}_{2} & S\\
}
$$
Let $\mathcal{E}^{(d)}$ be Laumon's sheaf on $P=X^{(d)}$. Also, $\textrm{det}(\mathcal{E})$ determines a rank $1$ local system on $\textrm{Pic}_{X}$, denote its fiber at $\Omega^{1}_{X}$ by $\mathcal{K}$. Then we have the following(Proposition $4.2.4$ and Remark $5.2$ of $[12]$):
\begin{Proposition}
\begin{enumerate1}
\item The Radon transform of $\mathcal{E}$ given by
$$R\check{\rho}_{*}\rho^{*}(\mathcal{E}^{(d)}[d])[d-g-1]$$
is an irreducible perverse sheaf on $\check{P}$.
\item The restriction of the perverse sheaf
$$\widetilde{K}_{\mathcal{E}}=\mathcal{K}\otimes R\check{\rho}_{*}\rho^{*}(\mathcal{E}^{(d)}[d])[d-g-1]$$
to $U$ descends down to $\textrm{Bun}_{2}$ and this gives the automorphic sheaf $\textrm{Aut}_{\mathcal{E}}$.
\end{enumerate1}
\end{Proposition}
Since we expect the Poincar\'e sheaf on the stack of Higgs bundles to be the classical limit of the automorphic sheaf, it should be possible to construct the Poincar\'e sheaf by modifying the Radon transform construction to the case of Higgs sheaves. Hence it is natural to consider the following(This diagram is contained in the appendix of $[12]$):
$$\xymatrix{
 & T^{*}_{Z}(\check{P}\times P) \ar[ld]_{\pi} \ar[rd]^{p} & \\
T^{*}\check{P} & & T^{*}P\\
}
$$
It is not hard to see that $T^{*}P$ classifies the data $(D,t_{D})$ where $D$ is a length $d$ subscheme of $X$, and $t_{D}\in H^{0}(\Omega^{1}_{X}\otimes O_{D})$. Hence if we denote the cotangent bundle of $X$ by $S$, $T^{*}P$ can be naturally identified as an open subscheme of $\HS^{d}$ by embedding $D$ into $S$ using the section $t_{D}$. Similarly, $T^{*}\check{P}$ classifies the data:
$$0\rightarrow \Omega^{1}_{X}\rightarrow L_{2}\rightarrow L_{1}\rightarrow 0, \varphi\in\textrm{Hom}(L_{2},L_{1}\otimes \Omega^{1}_{X})$$
Let $\mathcal{U}$ be the open subscheme of $T^{*}\check{P}$ satisfying the condition $H^{0}(\check{L}_{2}\otimes \Omega^{2}_{X})=0$. Also, let $\hig'$ be the stack classifying the data:
$$(E,\phi); 0\rightarrow \Omega^{1}_{X}\hookrightarrow E$$
where $(E,\phi)$ is a rank two Higgs bundle with value in $\Omega^{1}_{X}$, and $\Omega^{1}_{X}\hookrightarrow E$ is a subsheaf such that the quotient is a line bundle, and we require that $H^{0}(\check{E}\otimes\Omega^{2}_{X})=0$. Then arguing as Lemma~\ref{elementary properties of Y} we see that $\hig'$ is naturally a closed subscheme of $\mathcal{U}$. Projective duality implies that there are natural isomorphisms:
$$T^{*}\check{P}\backslash T^{*}S\times_{S}\check{P}\simeq T^{*}_{Z}(\check{P}\times P)\backslash T^{*}_{S}(S\times S)\times_{S}Z\simeq T^{*}P\backslash T^{*}S\times_{S}P$$
Hence it induces:
$$(T^{*}\check{P}\backslash T^{*}S\times_{S}\check{P})\times\hig\simeq(T^{*}P\backslash T^{*}S\times_{S}P)\times\hig$$
Moreover, if we denote the universal spectral curve by $\mathcal{C}$, then under this isomorphism, the image of $\textrm{Hilb}^{d}_{\mathcal{C}|H}\times_{H}\hig$ in the $(T^{*}P\backslash T^{*}S\times_{S}P)\times\hig$ corresponds to $(\hig'\cap(T^{*}P\backslash T^{*}S\times_{S}P))\times_{H}\hig$. The sheaf $Q$ on $\HS^{d}\times\hig=T^{*}P\times\hig$ obtained from Proposition~\ref{definition of Q} should be viewed as the analogue of the Laumon's sheaf in the Higgs setting. Since $Q$ is supported on $\textrm{Hilb}^{d}_{C|H}\times_{H}\hig$, via the isomorphisms above, $Q$ can also be viewed as a coherent sheaf on $(\hig'\cap(T^{*}P\backslash T^{*}S\times_{S}P))\times_{H}\hig$. In order the extend is to the entire $\hig'\times_{H}\hig$, it is natural to consider the full diagram:
$$\xymatrix{
 & T^{*}_{Z}(\check{P}\times P)\times\hig \ar[ld]_{\pi} \ar[rd]^{p} & \\
T^{*}\check{P}\times\hig & & T^{*}P\times\hig\\
}
$$
and ask whether $\pi_{*}(p^{*}(Q))$ is a Cohen-Macaulay sheaf supported on $\hig'\times_{H}\hig$. Also, to generalize this picture to Higgs bundles with values in other line bundles, we need to give a more intrinsic characterization of $T^{*}_{Z}(\check{P}\times P)$. It is not hard to show in this case that $T^{*}_{Z}(\check{P}\times P)$ can be identified with the blowup of $T^{*}\check{P}$ along the subscheme $T^{*}S\times_{S}\check{P}$. Hence all these motivates the following construction(See Construction~\ref{definition of Y} in subsection $3.3$)
\begin{construction}
Let $X$ be a smooth projective curve, $L$ a line bundle on $X$ with degree $l\geq 2g-1$, and $x_{0}$ a fixed point on $X$. Let $\mathcal{Y}$ be the stack classifying the data $(E,\varphi,s,\sigma)$ where $E$ is a rank two vector bundle on $X$ of degree $m$ with $H^{1}(E)=0$, $H^{0}(\check{E}\otimes L)=0$, and $E$ is globally generated, $s$ a nonzero global section of $E$ such that the quotient $M=E/O_{X}$ is a line bundle, and $\sigma$ is a trivialization of $M_{x_{0}}$, and $\varphi\in \textrm{Hom}(E,M\otimes L)$.
\end{construction}

We also need:

\begin{construction}
Let $\hig'^{m}$ be the moduli stack classifying the data $(E,\phi,s,\sigma)$ where $(E,s,\sigma)$ satisfying the same condition as in $\mathcal{Y}$, and $(E,\phi)$ is a Higgs bundle.
\end{construction}
In Lemma~\ref{elementary properties of Y} we will prove that $\hig^{'m}$ is a closed substack of $\mathcal{Y}$.
Now let us consider:
$$\xymatrix{
X\times\mathcal{Y} \ar[d]^{f}\\
\mathcal{Y}\\
}
$$
By the definition of $\mathcal{Y}$, we have the following morphism of vector bundles on $X\times\mathcal{Y}$:
$$E\xrightarrow{\varphi} M\otimes L$$
Hence this induced a section of $M\otimes L$ via:
$$O_{X}\xrightarrow{s} E\xrightarrow{\varphi} M\otimes L$$
So we get a global section of the vector bundle $f_{*}(M\otimes L)$ on $\mathcal{Y}$:
$$O\rightarrow f_{*}(M\otimes L)$$
by pushing it forward. Let $\mathcal{Z}$ be the vanishing locus of this section. $\mathcal{Z}$ is the analog of the closed subscheme $T^{*}S\times_{S}\check{P}\subseteq T^{*}\check{P}$. We will prove that $\mathcal{Z}$ is a smooth closed substack $\mathcal{Z}$ of $\mathcal{Y}$ and a complete intersection in $\mathcal{Y}$(See Lemma~\ref{definition of Z}). There exists a natural morphism $\mathcal{Y}\backslash\mathcal{Z}\rightarrow \HS^{d}$ defined by the following procedure:\\
We have a morphism $O_{\mathcal{Y}}\rightarrow f_{*}(M\otimes L)$ on $\mathcal{Y}$ which is nonvanishing over $\mathcal{Y}\backslash\mathcal{Z}$, hence we get a nonvanishing global section $t$ of $M\otimes L$ on $X\times(\mathcal{Y}\setminus\mathcal{Z})$. So if we denote the vanishing locus of $t$ by $D$, then $D$ is a closed substack of $X\times\mathcal{Y}$ and it is a family of finite subscheme of length $d$ of $X$ over $\mathcal{Y}\backslash\mathcal{Z}$. Notice that $t$ is given by the composition:
$$O_{X}\xrightarrow{s} E\rightarrow M\otimes L$$
If we restrict the above morphisms of vector bundles to $D$, we see that the composition:
$$O_{D}\xrightarrow{s_{D}} E\otimes O_{D}\rightarrow M\otimes L\otimes O_{D}$$
is equal to $0$ by the definition of $D$. Since we have the exact sequence of vector bundles:
$$0\rightarrow O_{X}\rightarrow E\rightarrow M\rightarrow 0$$
So we get:
$$0\rightarrow O_{D}\rightarrow E\otimes O_{D}\rightarrow M\otimes O_{D}\rightarrow 0$$
From this we see that the morphism:
$$E\otimes O_{D}\rightarrow M\otimes L\otimes O_{D}$$
factors through:
$$E\otimes O_{D}\rightarrow M\otimes O_{D}\rightarrow M\otimes L\otimes O_{D}$$
This gives a section $t_{D}\in H^{0}(L\otimes O_{D})$, and we embed $D$ into $S$ via $t_{D}$. So this defines a morphism $\mathcal{Y}\setminus\mathcal{Z}\rightarrow \HS^{d}$.\\
We want to resolve the rational map $\mathcal{Y}\backslash\mathcal{Z}\rightarrow \HS^{d}$. To do that we need the blowup of $\mathcal{Y}$ along $\mathcal{Z}$. Denote the blowup of $\mathcal{Y}$ along $\mathcal{Z}$ by $\mathcal{B}$. In Proposition~\ref{resolve the rational map} we will show that the morphism
$$\mathcal{Y}\backslash\mathcal{Z}\rightarrow \HS^{d}$$
extends to a morphism
$$\mathcal{B}\rightarrow \HS^{d}$$
So in the end we have the following diagram which will be called the fundamental diagram(The construction of the diagram is given in Proposition~\ref{resolve the rational map} of subsection $3.3$):
\begin{equation}\label{fundamental diagram}
\xymatrix{
 & \mathcal{B} \ar[r] \ar[d]^{\pi} & \HS^{d}\\
\mathcal{Z} \ar[r] & \mathcal{Y} & \\
}
\end{equation}
Let $\hig^{(n)}$ be the open substack of $\hig$ defined at the beginning of section $4$. The fundamental diagram above induces the following diagram:
$$\xymatrix{
\mathcal{B}\times\hig^{(n)} \ar[d]^{\pi} \ar[r]^{p} & \HS^{d}\times\hig^{(n)}\\
\mathcal{Y}\times\hig^{(n)}\\
}
$$
From the dimension calculations in Lemma~\ref{elementary properties of Y} and Lemma~\ref{definition of Z}, it follows easily that $\hig^{'m}\times_{H}\hig^{(n)}$ has codimension $d$ in $\mathcal{Y}\times\hig^{(n)}$ where $d=l+m$. The main theorem of the paper can be reformulated as follows:
\begin{theorem}\label{rewrite main theorem}
$\pi_{*}(p^{*}(Q))$(Here all functors are derived) is a Cohen-Macaulay sheaf of codimension $d$ supported on $\hig^{'m}\times_{H}\hig^{(n)}$.
\end{theorem}
Even though this theorem looks weaker than Theorem~\ref{the main theorem}, in section $5$ we will show that we can deduce Theorem~\ref{the main theorem} from Theorem~\ref{rewrite main theorem}. The proof of this theorem will be given in section $5$. It is obtained by more detailed study of the properties of the morphism $p$ in subsection $3.3$ and a cohomological calculation in section $4$.\\
\\
In the rest of this subsection, let us discuss the relationship between this construction and our previous construction of the Poincar\'e sheaf on the stack of Higgs bundles for $\p$ in $[13]$. Let us recall the following construction:
\begin{construction}
Let $\hig'$ be the moduli stack classifying the data $(E,\phi,s)$ where $(E,\phi)$ is a rank two Higgs bundle on $\p$ with value in $O(n)$ such that the underlying vector bundle $E$ is isomorphic to $O\oplus O$, and $s$ is a nonzero global section of $E$. Here we assume $n\geq 2$.
\end{construction}
The main result of our previous work can be summarized in the following theorem:
\begin{theorem}
\begin{enumerate1}
\item There exists a smooth closed substack $\mathcal{Z}$ of $\hig'$ of codimension $n+1$ such that $\mathcal{Z}$ is a complete intersection in $\hig'$.
\item There exists a morphism $\hig'\backslash\mathcal{Z}\rightarrow \HS^{n}$. Moreover, if we denote the blowup of $\hig'$ along $\mathcal{Z}$ by $\hig''$, then the morphism $\hig'\backslash\mathcal{Z}\rightarrow \HS^{n}$ extends to:
$$\xymatrix{
\hig'' \ar[d]^{\pi} \ar[r]^{g} & \HS^{n}\\
\hig'\\
}
$$
\item Consider the diagram:
$$\xymatrix{
\hig''\times_{H}\hig^{(-n)} \ar[d]^{\pi} \ar[r]^{g} & \HS^{n}\times\hig^{(-n)}\\
\hig'\times_{H}\hig^{(-n)}\\
}
$$
where $\hig^{(-n)}$ denotes the open substack of $hig^{-n}$ such that the underlying vector bundle is isomorphic to $O(-n)\oplus O(-n)$. Then $\pi_{*}(g^{*}(Q))$ is a maximal Cohen-Macaulay sheaf on $\hig'\times_{H}\hig^{(-n)}$
\end{enumerate1}
\end{theorem}

Now let us describe the relationship between this paper and our previous work on the Poincar\'e sheaf in the case of Higgs bundles of $\p$ in $[13]$. The construction in the previous work is similar to the construction in this paper, except that in the case of $\p$ we take the blowup of $\hig'$, while in this paper we take the blowup of $\mathcal{Y}$ instead of $\hig^{'m}$. It turns out that we can still consider the stack $\mathcal{Y}$ for $\p$. And  we still have the closed substack $\mathcal{Z}'$ of $\mathcal{Y}$ and $\hig'$ will be a closed substack of $\mathcal{Y}$. What makes $\p$ special is that the closed substack $\mathcal{Z}'$ and $\hig'$ inside $\mathcal{Y}$ are transversal to each other, hence their intersection
$$\mathcal{Z}=\mathcal{Z}'\cap\hig'$$
is still a complete intersection in $\hig'$. So the restriction of the blowup $\mathcal{B}$ to $\hig'$ is isomorphism to the blowup of $\hig'$ along $\mathcal{Z}$. This explains that in the case of $\p$ is suffices to consider the stack $\hig'$ and its blowup $\hig''$. In the case of higher genus curves this is no longer true, we need to work with $\mathcal{Y}$ in order to get reasonable objects.

\subsection{Acknowledgments}
I'm very grateful to Professor Dima Arinkin for introducing me to this fascinating subject as well as his encouragement and support throughout the entire project. This work is supported by NSF grant DMS-1603277.

\section{Preliminaries}
In this section we gather some preliminaries that will be needed in the later parts of the paper. Most of the materials in subsection $2.1$ through $2.3$ have been discussed in section $2$ of our previous work $[13]$, hence we will list the statements without proof.

\subsection{Blowup and cohomology}
Let $X$ be a scheme, $\mathcal{E}$ a vector bundle of rank $n+1$ on $X$ with a global section $s\in H^{0}(X,\mathcal{E})$, such that the vanishing locus of $s$ is a regular embedding of codimension $n+1$. Denote the vanishing locus of $s$ by $Z$. Then we have the following description about the blowup of $X$ along $Z$ (See Chapter $11$ of $[10]$):
\begin{Proposition}\label{blowup as regular embedding}

The blowup of $X$ along $Z$ is a regular embedding of codimension $n$ in $\textbf{P}(\mathcal{E})$:

$$
\xymatrix{
  Bl_{Z}X \ar[rr]^{i} \ar[dr]^{p}
                &  &    \textbf{P}(\mathcal{E})=Proj(Sym_{X}\check{\mathcal{E}}) \ar[dl]^{\pi}    \\
                & X
}
$$
described by the following: On $\mathbf{P}(\mathcal{E})$ we have a natural morphism of vector bundles:
$$\pi^{*}\mathcal{E}\xrightarrow{\varphi} T_{P(\mathcal{E})\mid X}\otimes O(-1)$$
The blowup is the vanishing locus of $\varphi(\pi^{*}s)\in H^{0}(T_{P(\mathcal{E})\mid X})$. It is a regular embedding in $P(\mathcal{E})$ so we have a Koszul resolution:
\begin{equation}\label{bl}
\bigwedge^{*}(\Omega_{P(\mathcal{E})\mid X}\otimes O(1))\rightarrow O_{Bl_{Z}X}
\end{equation}
In particular, $p$ is a local complete intersection morphism
\end{Proposition}

\begin{corollary}\label{bl resolve nonflatness}
Let $Y\xrightarrow{f} X$ be a proper flat morphism of schemes with geometrically integral fibers. Let $L$ be a line bundle on $Y$ such that $f_{*}(L)$ is a vector bundle of rank $n+1$ and the formation of $f_{*}(L)$ commutes with arbitrary base change. Let $s$ is a global section of $L$ such that the vanishing locus of the induced section $t$ of $f_{*}(L)$ has codimension $n+1$ on $X$. Let $Z$ be the vanishing locus of $t$, Set $Y'=Y\times_{X}Bl_{Z}X$. Consider:
\[\begin{CD}
Y'@>{g'}>>Y\\
@V{f'}VV@V{f}VV\\
Bl_{Z}X@>{g}>>X\\
\end{CD}\]
Then We have:
\begin{enumerate1}
\item The vanishing locus of $s$ on $Y$ is a relative effective Cartier divisor over the open subset $X\backslash Z$.
\item The section $s$ extends to a morphism $O(E)\xrightarrow{s'} g^{'*}(L)$ on $Y'$ such that the vanishing locus of $s'$ is a relative effective Cartier divisor over $Bl_{Z}X$ which extends the relative effective Cartier divisor over $X\backslash Z$
\end{enumerate1}
\end{corollary}

The following lemma will be useful when we compute cohomology of sheaves on blowups:
\begin{lemma}\label{pushforward of bl}
Keep the same assumption as above. Consider:
$$
\xymatrix{
  Bl_{Z}X \ar[rr]^{i} \ar[dr]^{p}
                &  &    \textbf{P}(\mathcal{E})=Proj(Sym_{X}\check{\mathcal{E}}) \ar[dl]^{\pi}    \\
                & X
}
$$
Suppose $K$ is an object in $D^{b}_{coh}(\mathbf{P}(\mathcal{E}))$ such that $R\pi_{*}(K\otimes O(a))\in D^{\leq 0}_{coh}(X)$ for $a\geq 0$, then we have $Rp_{*}Li^{*}K\in D^{\leq 0}_{coh}(X)$

\end{lemma}

The following property results from Grothendieck duality:
\begin{Proposition}\label{dualizing sheaf of bl}
Keep the same assumption as the beginning of this subsection. Let $E$ be the exceptional divisor in $\blA$. The dualizing sheaf $\omega_{\blA\mid A}\simeq O(nE)$.
\end{Proposition}

\subsection{Cohen-Macaulay sheaves}
In this section we shall review some facts about Cohen-Macaulay sheaves that will be used freely in the paper. For simplicity we shall work with a Gorenstein schemes $X$ (since all schemes(stacks) appearing in the paper are Gorenstein), so that the dualizing complex of $X$ can be taken to be $O_{X}$.
Most of these properties can be found in $[18]$.

\begin{definition}
Let $M$ be a coherent sheaf on $X$ such that $\textrm{codim} (\textrm{Supp(M)})=d$. Then $M$ is called Cohen-Macaulay of codimension $d$ if $\RH_{O_{X}}(M,O_{X})$ sits in degree $d$. In particular, $M$ is a Cohen-Macaulay module. $M$ is called maximal Cohen-Macaulay if $d=0$
\end{definition}
Notice that if $d=0$ then both $M$ and $\mathcal{H}om_{O_{X}}(M,O_{X})$ are maximal Cohen-Macaulay, and the functor $\mathcal{H}om_{O_{X}}(-,O_{X})$ induces an anti-auto-equivalence in the category of maximal Cohen-Macaulay sheaves.

\begin{Proposition}\label{codimension of supp of CM sheaves}
If $M$ is a Cohen-Macaulay sheaf of codimension $d$ on a Gorenstein scheme $X$, then the support of $M$ is of pure dimension without embedded components, and $\textrm{codim}(\textrm{Supp}(M))=d$. Moreover, the functor $M\rightarrow \mathcal{E}xt^{d}_{O_{X}}(M,O_{X})$ induces an anti-auto-equivalence of the category of Cohen-Macaulay sheaves of codimension $d$.
\end{Proposition}

We will use  the following property of maximal Cohen-Macaulay sheaves ($[1]$ Lemma $2.2$):
\begin{Proposition}\label{uniqueness of CM sheaf}
Let $X$ be a Gorenstein scheme of pure dimension, $M$ a maximal Cohen-Macaulay sheaf and $Z\subseteq X$ a closed subscheme of codimension greater than or equals to $2$, then we have $M\simeq j_{*}(M|_{X-Z})$ for $j$: $X-Z\hookrightarrow X$. Hence for any two maximal Cohen-Macaulay sheaves $M$ and $N$, we have $Hom_{X}(M,N)\simeq Hom_{X\backslash Z}(M|_{X\backslash Z},N|_{X\backslash Z})$

\end{Proposition}

\begin{Proposition}\label{flatness of CM sheaves}
Let $X$ be a Gorenstein scheme and $Y$ be a smooth scheme. Let $X\xrightarrow{f} Y$ be a flat morphism. If $M$ is a maximal Cohen-Macaulay sheaf on $X$, then $M$ is flat over $Y$

\end{Proposition}

\subsection{Hilbert scheme of points}
In this section we review some facts about Hilbert scheme of points of curves and surfaces that will be used in the proof of the main theorem. First let us look at Hilbert schemes of $X$. It is well known that the symmetric power $X^{(d)}$ is the Hilbert scheme of $X$ parameterizing finite subschemes of length $d$, and $X^{d}$ classifies finite subscheme of length $d$ together with a flag:
$${\O}=D_{0}\subseteq D_{1}\subseteq\cdots\subseteq D_{d}=D$$
such that $\textrm{ker}(O_{D_{i}}\rightarrow O_{D_{i-1}})$ has length $1$. The natural morphism:
$$X^{d}\xrightarrow{\eta} X^{(d)} \qquad ({\O}=D_{0}\subseteq D_{1}\subseteq\cdots\subseteq D_{d}=D)\rightarrow D$$

Now let us review some facts about Hilbert scheme of surfaces and planar curve. The following theorem is well-known (see $[9]$):
\begin{theorem}
Let $S$ be a smooth surface, then $\HS^{n}$ is smooth of dimension $2n$
\end{theorem}

We also set $\HS^{' d}$ be the open subscheme of $\HS^{d}$ parameterizing subschemes $D$ that can be embedded into a smooth curve (This notion is introduced in $[1]$). In the rest of the paper, $S$ is going to be the total space of a line bundle $L$ on $X$, and we shall work with a particular open subscheme of $\HS^{d}$ defined by the following proposition:
\begin{Proposition}\label{an open subscheme of HS}
View $X^{(d)}$ as the Hilbert scheme of $X$ parameterizing length $d$ subschemes. Let $D$ be the universal subscheme:
$$\xymatrix{
D \ar[rr] \ar[dr] & & X\times X^{(d)} \ar[ld]^{\pi}\\
 & X^{(d)}\\
}
$$
Then the total space of the vector bundle $\pi_{*}(L_{D})$ can be naturally identified with an open subscheme of $\HS^{d}$. Moreover, it is contained in the open subscheme $\HS^{' d}$.

\end{Proposition}

\begin{proof}
Observe that if $D$ is length $d$ subscheme of $X$ and $s$ a global section of $L_{D}$, then we can embed $D$ into $S$ using the section $s$. This gives $D$ the structure of a closed subscheme of $S$. Now the assertion follows immediately from this.

\end{proof}

\begin{corollary}\label{Hilbert scheme of P1 and S}
Keep the same notation as the previous Proposition. Let $V$ denote the total space the vector bundle $\pi_{*}(L_{D})$ on $X^{(d)}$, and denote its pullback to $X^{d}$ by $V'$.
\begin{enumerate1}
\item there exist an open embedding $V\hookrightarrow \HS^{d}$ with image contained in $\HS^{' d}$.
\item There are Cartesian diagrams:
$$\xymatrix{
V' \ar[r] \ar[d]^{\psi} & \Fl^{' d} \ar[d]\\
V \ar[r] & \HS^{' d}
}
$$
$$\xymatrix{
V' \ar[r]^{r'} \ar[d]^{\psi} & X^{d} \ar[d]^{h}\\
V \ar[r]^{r} & X^{(d)}\\
}
$$
\item The composition of the morphisms:
$$V'\subseteq \Fl^{' d}\xrightarrow{\sigma} S^{d}\rightarrow X^{d}$$
is the natural projection $V'\xrightarrow{r'} X^{d}$
\item Let $D'$ be the universal subscheme of $S$ over $\HS^{' n}$ and $D$ be the universal subscheme of $X$ over $X^{(d)}$. Consider $O_{D'}$ and $O_{D}$ as vector bundles on $\HS^{' n}$ and $X^{(d)}$. Let $r$ be the natural projection $V\xrightarrow{r}X^{(d)}$. Then we have: $\de(O_{D'})\simeq r^{*}(\de(O_{D}))$ over $V$
\end{enumerate1}
\end{corollary}

\begin{proof}
The identification of $V$ as an open subscheme of $\HS^{' d}$ is defined as follows. A point of $V$ corresponds to pairs $(D,s)$ where $D$ is a closed subscheme of length $d$ of $X$ and $s\in H^{0}(L_{D})$. Since the surface $S$ is the total space of the line bundle $L$, we can embed $D$ into $S$ using the section $s$. So this defines a morphism $V\rightarrow \HS^{' d}$. Now let us consider:
\[\begin{CD}
X\times X^{(d)}\\
@V{p}VV\\
X^{(d)}\\
\end{CD}\]
In Proposition~\ref{an open subscheme of HS} we proved that $V$ can be identified with an open subscheme of $\HS^{d}$ contained in $\HS^{' d}$. This proves $(1)$\\
For part $(2)$, notice that over $V$, the subscheme $D$ of $S$ comes from a subscheme of $X$, hence giving a flag of $D$ as a subscheme of $S$ is the same thing as giving a flag of $D$ as a subscheme of $X$, and we know that $Flag^{d}_{X}\simeq X^{d}$. This also proves part $(3)$\\
Part $(4)$ follows almost immediately from the fact that over $V$, the subscheme $D'$ of $S$ comes from the subscheme $D$ of $X$. Hence the assertion follows.

\end{proof}

\subsection{Jacobian of smooth curves and Fourier-Mukai transform}
In this section we briefly review some well-known facts about the Jacobian of smooth curves. Let $X$ be a smooth projective curve of genus $g$ with a fixed point $x_{0}$. Then the Jacobian $\Pd$ of $X$ parameterizes degree $d$ line bundles on $X$ with a trivialization at $x_{0}$. It is well known that we have the Abel-Jacobian map:
$$X^{(d)}\xrightarrow{A} \Pd$$
which is smooth when $d\geq 2g$. The following summarizes the properties we need:
\begin{lemma}\label{properties of abel-jacobian map}
Consider the following diagram:
$$\xymatrix{
X\times X^{d} \ar[r]^{h'} \ar[d]^{q''} & X\times X^{(d)} \ar[r]^{A'} \ar[d]^{q'} & X\times\Pd \ar[d]^{q}\\
X^{d} \ar[r]^{h} \ar@/_2pc/[rr]^{f} & X^{(d)} \ar[r]^{A} & \Pd\\
}
$$
Let $\mathcal{L}$ be the universal line bundle on $X\times\Pd$, then we have:
\begin{enumerate1}
\item $A'^{*}(\mathcal{L})\simeq O_{X\times X^{(d)}}(D)\otimes q'^{*}(O(-x_{0})^{(d)})$ where $D$ is the universal divisor on $X\times X^{(d)}$
\item $X^{(d)}$ is a projective bundle over $\Pd$ associated with the vector bundle $q_{*}(\mathcal{L})$, and $O(1)_{X^{(d)}|\Pd}\simeq O(x_{0})^{(d)}$
\item Let $\Delta_{ij}$ be the divisor on $X^{d}$ given by $x_{i}=x_{j}$. Then $h^{*}(\de(q'_{*}(O_{D}))^{-1})\simeq O(\sum_{i<j}\Delta_{ij})$
\item Let $\Theta$ be the theta divisor on $\Pd$. We have:
$$f^{*}O(\Theta)\simeq \Omega_{X}((d-g+1)x_{0})^{\boxtimes d}\otimes O(-\sum_{i<j}\Delta_{ij})$$
\item The dualizing sheaf $\omega_{X^{d}|X^{(d)}}\simeq O(\sum_{i<j}\Delta_{ij})$.

\end{enumerate1}
\end{lemma}

\begin{proof}
For part $(1)$, since $X^{(d)}$ parameterizes degree $d$ effective divisors of $X$, we have a universal divisor $D\hookrightarrow X\times X^{(d)}$ and the corresponding line bundle $O_{X}(D)$ on $X\times X^{(d)}$. Let $\mathcal{M}$ be the line bundle on $X^{(d)}$ given by $O_{X}(D)|_{x_{0}\times X^{(d)}}$. Then by pulling back to $X^{d}$, it is easy to see that $\mathcal{M}\simeq O(x_{0})^{(d)}$. The morphism $X^{(d)}\xrightarrow{A} \Pd$ is given by
$$[D]\in X^{(d)}\rightarrow O_{X}(D)\otimes(\mathcal{M}^{-1}|_{[D]})$$
Hence $A^{'*}(\mathcal{L})\simeq O_{X}(D)\otimes q^{'*}(O(-x_{0})^{(d)})$. This establishes $(1)$.\\
For $(2)$, notice that by $(1)$, we have a natural injection of line bundles:
$$q'^{*}(O(-x_{0})^{(d)})\hookrightarrow A'^{*}(\mathcal{L})\simeq O_{X\times X^{(d)}}(D)\otimes q'^{*}(O(-x_{0})^{(d)})$$
Hence $O(-x_{0})^{(d)}$ is naturally a subbundle of $q'_{*}(A'^{*}(\mathcal{L}))$ on $X^{(d)}$:
$$O(-x_{o})^{(d)}\hookrightarrow q'_{*}(A'^{*}(\mathcal{L}))$$
So the claim follows easily from this.\\
For $(3)$, notice that we have the following short exact sequences on $X\times X^{d}$ for each $i$:
$$0\rightarrow O_{X}(-\Delta_{1}-\cdots-\Delta_{i})\rightarrow O_{X}(-\Delta_{1}-\cdots\Delta_{i-1})\rightarrow O_{\Delta_{i}}(-\Delta_{1}-\cdots-\Delta_{i-1})\rightarrow 0$$
And from this it is easy to see that
$$h^{*}(\textrm{det}(q'_{*}(O_{D}))^{-1})\simeq \otimes^{i}q'_{*}(O_{\Delta_{i}}(\Delta_{1}+\cdots+\Delta_{i-1}))\simeq O(\sum_{i<j}\Delta_{ij})$$
For $(4)$, recall that in our setup, the theta divisor can be identified with:
$$\textrm{det}(q_{*}(\mathcal{L}(-(d-g+1)x_{0})))^{-1}$$
here $q_{*}$ stands for the derived pushforward. Hence from part $(1)$ we see that $f^{*}(O(\Theta))$ is given by:
$$\textrm{det}(q''_{*}(O_{X}(\Delta_{1}+\cdots+\Delta_{d}-(d-g+1)x_{0})))^{-1}$$
We have the exact sequence
$$0\rightarrow O_{X}(-(d-g+1)x_{0})\rightarrow O_{X}(D-(d-g+1)x_{0})\rightarrow O_{D}(D-(d-g+1)x_{0})\rightarrow 0$$
where $D$ stands for the divisor $\Delta_{1}+\cdots+\Delta_{d}$ on $X\times X^{d}$. So we have
$$\textrm{det}(q''_{*}(O_{X}(\Delta_{1}+\cdots+\Delta_{d}-(d-g+1)x_{0})))^{-1}\simeq\textrm{det}(q''_{*}(O_{D}(D-(d-g+1)x_{0})))^{-1}$$
Notice that for each $i$, we have the exact sequence:
$$0\rightarrow O_{X}(\Delta_{1}+\cdots+\Delta_{i-1})\rightarrow O_{X}(\Delta_{1}+\cdots+\Delta_{i})\rightarrow O_{\Delta_{i}}(\Delta_{1}+\cdots+\Delta_{i})\rightarrow 0$$
And it is not hard to see that
$$\textrm{det}(q''_{*}(O_{D}(D-(d-g+1)x_{0})))^{-1}\simeq\otimes^{i}(q''_{*}(O_{\Delta_{i}}(\Delta_{1}+\cdots+\Delta_{i}-(d-g+1)x_{0})))^{-1}$$
Hence the assertion follows from this.\\
For $(5)$, it is well known that the tangent sheaf of $X^{(d)}$ can be identified with $q'_{*}(O_{D}(D))$. Since we have exact sequences
$$0\rightarrow O_{X}(\Delta_{1}+\cdots+\Delta_{i-1})\rightarrow O_{X}(\Delta_{1}+\cdots+\Delta_{i})\rightarrow O_{\Delta_{i}}(\Delta_{1}+\cdots+\Delta_{i})\rightarrow 0$$
the assertion follows easily from this.

\end{proof}

Recall that if $J=\textrm{Pic}^{0}$ is the Jacobian of $X$, then there exists a Poincar\'e line bundle $\mathcal{P}$ on $J\times J$. Moreover, it is well-known that $\mathcal{P}$ induces an equivalence of the derived category of $J$ via the Fourier-Mukai transform:
$$\mathcal{F}\rightarrow Rp_{1 *}(p_{2}^{*}(\mathcal{F})\otimes\mathcal{P})$$
The inverse is given by:
$$\mathcal{G}\rightarrow Rp_{2*}(p_{1}^{*}(\mathcal{G})\otimes\mathcal{P}^{-1})[g]$$
The following lemma is a direct application of the Fourier-Mukai transform:
\begin{lemma}\label{vanishing of pushforward}
Let $K\in D^{b}_{coh}(X^{d})$ such that for any degree zero line bundle $L$ on $X$, we have $R^{i}\Gamma(K\otimes L^{\boxtimes d})=0$ for $i>p$. Then we have $Rf_{*}K\in D^{\leq p}_{coh}(Pic^{d})$ where $f$ is the natural morphism $X^{d}\xrightarrow{f} \Pd$
\end{lemma}

\begin{proof}
From the construction of $\mathcal{P}$, it is easy to check that the conditions on $K$ implies that the Fourier-Mukai transform of $Rf_{*}(K)$ belongs to $D^{\leq p}(\Pd)$, so the claim follows immediately from the fact that $Rp_{2 *}$ has cohomological dimension $g$ and the expression for the inverse transform.
\end{proof}

\section{Moduli of rank $2$ Higgs bundles on $X$}
Most of the materials in subsection $3.1$ and $3.2$ are well-known and have already appeared in one form or the other in our previous work $[13]$, we include them here for reader's convenience. On the other hand, subsection $3.3$ is of vital importance for the construction of the Poincar\'e sheaf. So the reader is recommended to jump directly to $3.3$ and return to $3.1$ and $3.2$ when necessary.

\subsection{Generalities on the geometry of the moduli of Higgs bundles}
In this subsection we review some general definitions and results about the geometry of the stack of Higgs bundles on $X$. We fix an integer $l\geq 2g-1$ and $L$ be a line bundle on $X$ of degree $l$.
\begin{definition}
A rank $2$ $L$-valued Higgs bundle on $X$ is a pair $(E,\phi)$ where $E$ is a rank $2$ vector bundle and the Higgs field $\phi$ is a morphism of vector bundles $E\xrightarrow{\phi} E\otimes L$

\end{definition}

We denote the stack of Higgs bundles by $\hig$. It decomposes as a disjoint union:
$$\hig=\coprod_{k}\hig^{k}$$
where $\hig^{k}$ is the stack of Higgs bundles of degree $k$. It is well-known that $\hig$ is an algebraic stack locally of finite presentation.
\begin{definition}
A Higgs bundle $(E,\phi)$ is called semistable if the following condition hold: For any line subbundle $E_{0}\subseteq E$ that is preserved by the Higgs field $\phi$ in the sense that $\phi(E_{0})\subseteq E_{0}\otimes L$, we have $\textrm{deg}(E_{0})\leq \dfrac{\textrm{deg}(E)}{2}$.

\end{definition}

Semistable Higgs bundles form an open substack, denote it by $\hig_{ss}$.

\begin{Proposition}\label{smoothness of semistable higgs bundles}
$\hig_{ss}$ is smooth of dimension $4l$.
\end{Proposition}

\begin{definition}\label{generically regular Higgs field}
A Higgs field $\phi$ is called generically regular if $\phi$ is a regular element in $\mathfrak{gl}_{2}(k(\eta))$ after trivializing $E$ and $L$ at the generic point $\eta$ of $X$. Higgs bundles with generically regular Higgs field forms an open substack of $\hig$, denote it by $\widetilde{\hig}$
\end{definition}

\subsection{Hitchin fibration and spectral curve}
In this subsection we are going to review certain properties of the Hitchin fibration that will be used later in the paper. Part $(2)$ of Proposition~\ref{properties of Hitchin fibration} will be frequently used. Lemma~\ref{resolution} will be used in Section $4$ to construct resolutions of sheaves. Most of the properties are well-known(See $[12]$ and the appendix of $[3]$ for a summary).

\begin{definition}
The Hitchin base $H$ is the affine space given by:
$$H^{0}(X,L)\times H^{0}(X,L^{2})$$

\end{definition}

It is well-known that we have the Hitchin fibration:
$$\hig\xrightarrow{h} H \qquad (E,\phi)\rightarrow (\textrm{tr}(\phi),\textrm{det}(\phi))$$
The following proposition summarizes the properties of Hitchin fibration:
\begin{Proposition}\label{properties of Hitchin fibration}
\begin{enumerate1}
\item $H$ is an affine space of dimension $3l-2(g-1)$
\item $h$ is a relative complete intersection morphism with fiber dimension $l+2(g-1)$
\end{enumerate1}

\end{Proposition}

Let $S$ be the surface defined by the total space of the line bundle $L$. We have:
$$\xymatrix{
S=\textrm{Spec}(\textrm{Sym}(L^{-1})) \ar[d]^{q}\\
X\\
}
$$
Any Higgs bundle can be naturally viewed as a coherent sheaf on the surface $S$, and we have the following lemma:

\begin{lemma}\label{resolution}
Let $(E,\phi)$ be a Higgs bundle. There is a locally free resolution of $E$ as a coherent sheaf on $S$:
$$0\rightarrow q^{*}(E)\otimes q^{*}(L^{-1})\rightarrow q^{*}(E)\rightarrow E\rightarrow 0$$
\end{lemma}

\begin{proof}
By the definition of $S$, there is a tautological section $T$ of $q^{*}(L)$ on $S$, and the morphism
$$q^{*}(E)\otimes q^{*}(L^{-1})\rightarrow q^{*}(E)$$
is given by $T-q^{*}(\phi)$. It is easy to see that this is a resolution of $E$ as an $O_{S}$ module.

\end{proof}

The following simple lemma will be used later
\begin{lemma}\label{a property of semistable higgs bundle}
Let $(E,\phi)$ be a semistable Higgs bundle such that the underlying vector bundle $E$ is not semistable. Then $(E,\phi)\in\widetilde{\hig}$.
\end{lemma}
\begin{proof}
Because $E$ is not semistable as a vector bundle, we can choose a line subbundle $\mathcal{M}\hookrightarrow E$ such that
$$\textrm{deg}(\mathcal{M})>\dfrac{\textrm{deg}(E)}{2}$$
We may assume $\textrm{deg}(E)$ is sufficiently large so that we have a global section $O_{X}\xrightarrow{s}\mathcal{M}$. Since $(E,\phi)$ is semistable as a Higgs bundle, the Higgs field cannot fix $\mathcal{M}$, hence $s$ and $\phi(s)$ are linearly independent at the generic point of $X$, hence $(E,\phi)\in\widetilde{\hig}$.

\end{proof}

\begin{lemma}\label{description of the dual}
Let $(E,\phi)$ be a rank two Higgs bundle on $X$, viewed as a coherent sheaf on the corresponding spectral curve $C$. Then $\Hom_{O_{C}}(E,O_{C})$ is isomorphic to $\check{E}\otimes L^{-1}$ as a Higgs bundle on $X$, with the Higgs field induced from $\phi$.
\end{lemma}

\subsection{Higgs bundles and Hilbert schemes}
In this subsection we will define and study the geometric properties of the stacks mentioned in subsection $1.4$. In particular we will construct Diagram \ref{fundamental diagram} mentioned in subsection $1.4$:
$$
\xymatrix{
 & \mathcal{B} \ar[r] \ar[d]^{\pi} & \HS^{d}\\
\mathcal{Z} \ar[r] & \mathcal{Y} & \\
}
$$
where $\mathcal{Y}$ is a smooth stack, $\mathcal{Z}$ is a smooth closed substack of $\mathcal{Y}$ and $\mathcal{B}$ is the blowup of $\mathcal{Y}$ along $\mathcal{Z}$. The construction is given in Proposition~\ref{resolve the rational map}. After that we will give some dimension estimates and a few further properties of the fundamental diagram which will be used in section $4$ and section $5$.\\

Let $X$ be a smooth projective curve of genus $g$. Fix a point $x_{0}$ in $X$. Let $l$ and $m$ be integers such that $m\geq 4g$, $l\geq 2g-1$ and $m>l$. Let $L$ be a fixed line bundle on $X$ of degree $l$. First let us define the main objects of study.

\begin{construction}
Let $\mathcal{X}$ be the stack classifying the data $(E,s,\sigma)$ where $E$ is a rank two vector bundle on $X$ of degree $m$ with $H^{1}(E)=0$, $H^{0}(\check{E}\otimes L)=0$, and $E$ is globally generated, $s$ a nonzero global section of $E$ such that the quotient $M=E/O_{X}$ is a line bundle, $\sigma$ is a trivialization of $M_{x_{0}}$.
\end{construction}

\begin{construction}\label{definition of Y}
Let $\mathcal{Y}$ be the stack classifying the data $(E,\varphi,s,\sigma)$ where $(E,s,\sigma)$ satisfying the same condition as $\mathcal{X}$, and $\varphi\in \textrm{Hom}(E,M\otimes L)$.
\end{construction}

\begin{construction}
Let $\hig'^{m}$ be the moduli stack classifying the data $(E,\phi,s,\sigma)$ where $(E,s,\sigma)$ satisfying the same condition as in $\mathcal{X}$, and $(E,\phi)$ is a Higgs bundle.
\end{construction}

We have the following lemma:
\begin{lemma}\label{elementary properties of Y}
\begin{enumerate1}
\item $\mathcal{Y}$ is naturally a vector bundle over $\mathcal{X}$
\item $\hig'^{m}$ is a closed substack of $\mathcal{Y}$
\item $\mathcal{Y}$ is smooth of dimension $2l+2m+1$
\item The natural morphism $\hig'^{m}\rightarrow\hig$ is smooth of dimension $m+2(1-g)+1$
\item $\hig'^{m}$ has dimension $4l+m+2(1-g)+1$
\end{enumerate1}
\end{lemma}

\begin{proof}
For $(1)$, notice that there is a natural morphism $\mathcal{Y}\rightarrow\mathcal{X}$ given by:
$$(E,\varphi,s,\sigma)\rightarrow (E,s,\sigma)$$
And the condition on $l$ and $m$ guarantees that $H^{1}(\check{E}\otimes M\otimes L)=0$. Hence the assertion follows from this.\\
For $(2)$, observe that the morphism $E\rightarrow M$ induced a natural morphism:
$$\textrm{Hom}(E,E\otimes L)\rightarrow \textrm{Hom}(E,M\otimes L)$$
Hence we have a morphism $\hig'^{m}\rightarrow \mathcal{Y}$ given by:
$$(E,\phi,s,\sigma)\rightarrow (E,\varphi,s,\sigma)$$
by sending $\phi$ to its image in $\textrm{Hom}(E,M\otimes L)$. We have an exact sequence:
$$0\rightarrow \check{E}\otimes L\rightarrow \check{E}\otimes E\otimes L\rightarrow \check{E}\otimes M\otimes L\rightarrow 0$$
Now consider:
$$\xymatrix{
X\times\mathcal{Y} \ar[d]^{f}\\
\mathcal{Y}\\
}
$$
The assumption $H^{0}(\check{E}\otimes L)=0$ implies that $R^{1}f_{*}(\check{E}\otimes L)$ is a vector bundle on $\mathcal{Y}$. From the definition of $\mathcal{Y}$, $\varphi$ induces a global section of $f_{*}(\check{E}\otimes M\otimes L)$:
$$O_{\mathcal{Y}}\xrightarrow{\varphi} f_{*}(\check{E}\otimes M\otimes L)$$
and it is easy to see that $\hig'^{m}$ is the vanishing locus of the morphism:
$$O_{\mathcal{Y}}\xrightarrow{\varphi}f_{*}(\check{E}\otimes M\otimes L)\rightarrow R^{1}f_{*}(\check{E}\otimes L)$$. So $(2)$ follows from this.\\
The rest of the lemma follows easily from the fact that the stack of rank two vector bundles on $X$ is smooth of dimension $4(g-1)$ and our assumption that $H^{1}(E)=0$.

\end{proof}

Next let us consider:
$$\xymatrix{
X\times\mathcal{Y} \ar[d]^{f}\\
\mathcal{Y}\\
}
$$
By the definition of $\mathcal{Y}$, we have the following morphism of vector bundles on $X\times\mathcal{Y}$:
$$E\xrightarrow{\varphi} M\otimes L$$
Hence this induced a section of $M\otimes L$ via:
$$O_{X}\xrightarrow{s} E\xrightarrow{\varphi} M\otimes L$$
So we get:
$$O\rightarrow f_{*}(M\otimes L)$$
on $\mathcal{Y}$ by pushing it forward.
\begin{lemma}\label{definition of Z}
Let $\mathcal{Z}$ be the vanishing locus of the morphism $O\rightarrow f_{*}(M\otimes L)$ on $\mathcal{Y}$. Then $\mathcal{Z}$ is naturally a vector bundle over $\mathcal{X}$ and it is a subbundle of $\mathcal{Y}$. $\mathcal{Z}$ has dimension $m+l+g$.
\end{lemma}
\begin{proof}
Consider:
$$\xymatrix{
X\times\mathcal{Y} \ar[d]^{f}\\
\mathcal{Y}\\
}
$$
The natural exact sequence:
$$0\rightarrow O_{X}\rightarrow E\rightarrow M\rightarrow 0$$
induces:
$$0\rightarrow\check{M}\rightarrow\check{E}\rightarrow O_{X}\rightarrow 0$$
Hence we get:
$$0\rightarrow L\rightarrow\check{E}\otimes M\otimes L\rightarrow M\otimes L\rightarrow 0$$
Now $\varphi$ can be identified with a section of
$$f_{*}(\check{E}\otimes M\otimes L)$$
and the section $O\rightarrow f_{*}(M\otimes L)$ is given by the image of $\varphi$ in $f_{*}(M\otimes L)$. From this it is easy to see that $\mathcal{Z}$ can be identified with the vector bundle over $\mathcal{X}$ given by $g_{*}(L)$ where $g$ is the morphism:
$$\xymatrix{
X\times\mathcal{X} \ar[d]^{g}\\
\mathcal{X}\\
}
$$

\end{proof}
Set $d=l+m$. The following observation is central to the construction of the Poincar\'e sheaf:
\begin{Proposition}\label{resolve the rational map}
Let $S$ be the smooth surface given by the total space of the line bundle $L$, and $\mathcal{B}$ be the blowup of $\mathcal{Y}$ along the closed substack $\mathcal{Z}$. Then:
\begin{enumerate1}
\item There exists a natural morphism $\mathcal{Y}\setminus\mathcal{Z}\xrightarrow{p} \HS^{d}$ such that its image is contained in the open subscheme $V$ of $\HS^{d}$ defined in Corollary~\ref{Hilbert scheme of P1 and S}.
\item The morphism $\mathcal{Y}\setminus\mathcal{Z}\xrightarrow{p} \HS^{d}$ extends to a morphism of stacks:
$\mathcal{B}\rightarrow \HS^{d}$ such that the image of $\mathcal{B}$ is contained in the open subscheme $V$.
\item The morphism $\mathcal{Y}\setminus\mathcal{Z}\xrightarrow{p} V\subseteq\HS^{d}$ is smooth.

\end{enumerate1}
Hence we have the following diagram:
$$\xymatrix{
\mathcal{B} \ar[d]^{\pi} \ar[r] & \HS^{d}\\
\mathcal{Y}
}
$$
\end{Proposition}

\begin{proof}
For $(1)$, consider:
$$\xymatrix{
X\times\mathcal{Y} \ar[d]^{f}\\
\mathcal{Y}\\
}
$$
By Lemma~\ref{definition of Z}, we have a morphism $O_{\mathcal{Y}}\rightarrow f_{*}(M\otimes L)$ on $\mathcal{Y}$ which is nonvanishing over $\mathcal{Y}\setminus\mathcal{Z}$, hence we get a nonvanishing global section $t$ of $M\otimes L$ on $X\times(\mathcal{Y}\setminus\mathcal{Z})$. So if we denote the vanishing locus of $t$ by $D$, then $D$ is a closed substack of $X\times(\mathcal{Y}\setminus\mathcal{Z})$ and it is a family of finite subscheme of length $d$ of $X$ over $\mathcal{Y}\setminus\mathcal{Z}$. Notice that $t$ is given by the composition:
$$O_{X}\xrightarrow{s} E\rightarrow M\otimes L$$
If we restrict the above morphisms of vector bundles to $D$, we see that the composition:
$$O_{D}\xrightarrow{s_{D}} E\otimes O_{D}\rightarrow M\otimes L\otimes O_{D}$$
is equal to $0$ by the definition of $D$. Since we have the exact sequence of vector bundles:
$$0\rightarrow O_{X}\rightarrow E\rightarrow M\rightarrow 0$$
So we get:
$$0\rightarrow O_{D}\rightarrow E\otimes O_{D}\rightarrow M\otimes O_{D}\rightarrow 0$$
From this we see that the morphism:
$$E\otimes O_{D}\rightarrow M\otimes L\otimes O_{D}$$
factors through:
$$E\otimes O_{D}\rightarrow M\otimes O_{D}\rightarrow M\otimes L\otimes O_{D}$$
This gives a section $t_{D}\in H^{0}(L\otimes O_{D})$, and we embed $D$ into $S$ via $t_{D}$. So this defines a morphism $\mathcal{Y}\setminus\mathcal{Z}\rightarrow \HS^{d}$\\
For $(2)$, consider:
$$\xymatrix{
X\times\mathcal{B} \ar[r]^{\pi'} \ar[d]^{f'} & X\times\mathcal{Y} \ar[d]^{f}\\
\mathcal{B} \ar[r]^{\pi} & \mathcal{Y}\\
}
$$
Now apply Corollary~\ref{bl resolve nonflatness}, we see that the section:
$$O\rightarrow M\otimes L$$
extends to:
$$O(E)\rightarrow \pi'^{*}(M\otimes L)$$
on $X\times\mathcal{B}$, and its vanishing locus $D'$ defines a family of finite subscheme of length $d$ of $X$ over $\mathcal{B}$. Now since $D'$ is a subscheme of the pullback of $D$ to $X\times\mathcal{B}$, so the section $t_{D}\in H^{0}(L\otimes O_{D})$ induces a section of $H^{0}(L\otimes O_{D'})$, and we embed $D'$ into $S$ using this section. This gives a morphism $\mathcal{B}\rightarrow \HS^{d}$ which extends the morphism $\mathcal{Y}\rightarrow \HS^{d}$ in $(1)$.\\
For $(3)$, we claim that the fibers of $\mathcal{Y}\setminus\mathcal{Z}\rightarrow \HS^{d}$ is a $G_{m}$ torsor over its image. In fact, let $D$ be vanishing locus of the section $t$ of $M\otimes L$ as in $(1)$. So $M\otimes L\simeq O_{X}(D)$. By the definition of $\mathcal{Y}$, we have morphisms
$$E\rightarrow M\otimes L \qquad 0\rightarrow O_{X}\rightarrow E\rightarrow M\rightarrow 0$$
on $X\times\mathcal{Y}$. They induces
$$E\otimes M^{-1}\rightarrow L \qquad 0\rightarrow M^{-1}\rightarrow E\otimes M^{-1}\rightarrow O_{X}\rightarrow 0$$
By the definition of $D$, the composition
$$M^{-1}\rightarrow E\otimes M^{-1}\rightarrow L\rightarrow L\otimes O_{D}$$
is zero. Hence
$$E\otimes M^{-1}\rightarrow L\rightarrow L\otimes O_{D}$$
factors as
$$E\otimes M^{-1}\rightarrow O_{X}\rightarrow L\otimes O_{D}$$
Since
\begin{equation}\label{exact sequence}
0\rightarrow M^{-1}\rightarrow E\otimes M^{-1}\rightarrow O_{X}\rightarrow 0
\end{equation}
is an exact sequence, so we get the following morphism of exact sequences:
\begin{equation}\label{commu}
\xymatrix{
0 \ar[r] & M^{-1} \ar[r] \ar[d] & E\otimes M^{-1} \ar[r] \ar[d] & O_{X} \ar[r] \ar[d] & 0\\
0 \ar[r] & L(-D) \ar[r] & L \ar[r] & L\otimes O_{D} \ar[r] & 0\\
}
\end{equation}
The exact sequence~\ref{exact sequence} induces morphism $H^{0}(O_{X})\rightarrow H^{1}(M^{-1})$. $E$ gives a class in $\textrm{Ext}^{1}(O_{X},M^{-1})=H^{1}(M^{-1})$ which can be identified with the image of $1\in H^{0}(O_{X})$ in $H^{1}(M^{-1})$. Because of the commutative diagram~\ref{commu}, it can also be identified with the image of $t_{D}\in H^{0}(L\otimes O_{D})$ under the morphism
$$H^{0}(L\otimes O_{D})\rightarrow H^{1}(L(-D))=H^{1}(M^{-1})$$
since $M\otimes L\simeq O_{X}(D)$. Hence given any point $(D,t_{D})\in V$ which lies in the image of $\mathcal{Y}\setminus\mathcal{Z}$, if we choose an isomorphism $M\simeq L^{-1}(D)$, we can recover the corresponding point $(E,\varphi,s,\sigma)$ of $\mathcal{Y}$ by pulling back the exact sequence via $t_{D}$:
$$\xymatrix{
0 \ar[r] & L(-D)\simeq M^{-1} \ar[d] \ar[r] & E\otimes L(-D)\simeq E\otimes M^{-1} \ar[d] \ar[r] & O_{X} \ar[d]^{t_{D}} \ar[r] & 0\\
0\ar[r] & L(-D) \ar[r] &  L \ar[r] & L\otimes O_{D} \ar[r] & 0\\
}
$$
Hence the assertion follows from this.\\

\end{proof}

In the next lemma we will give a description about the preimage of $\textrm{Hilb}^{d}_{C}\subseteq\HS^{d}$ under the morphism $p$ and some dimension estimates:

\begin{lemma}\label{further props of B}
\begin{enumerate1}
\item Fix a spectral curve $C$. Consider
$$\mathcal{Y}\backslash\mathcal{Z}\xrightarrow{p}\HS^{d}$$
Then $p^{-1}(\textrm{Hilb}_{C}^{d})\cap(\mathcal{Y}\backslash\mathcal{Z})=\hig'^{m}_{C}$ where $\hig'^{m}_{C}$ stands for the closed substack of $\hig'^{m}$ consists of Higgs bundles with spectral curve $C$.
\item Consider the following diagram
$$\xymatrix{
\mathcal{B}\times\hig \ar[d]^{\pi} \ar[r]^{p} & \HS^{d}\times\hig\\
\mathcal{Y}\times\hig\\
}
$$
The intersection of $\mathcal{Z}\times\hig$ with $\pi(p^{-1}(\textrm{Hilb}^{d}_{\mathcal{C}|H}\times_{H}\hig))$ has dimension less than or equals to $m+3l+2g-1$.

\end{enumerate1}
\end{lemma}

\begin{proof}
For $(1)$, first notice that we always have $p(\hig'^{m}_{C})\subseteq\textrm{Hilb}^{d}_{C}$ by the definition of $p$. To prove the reverse inclusion, take a point $(D,t_{D})\in V$ where $D$ is a degree $d$ divisor on $X$ and $t_{D}\in H^{0}(L\otimes O_{D})$. Then we can recover its preimage in the following way. Because $D$ is a closed subscheme of $S$, we have an exact sequence:
$$0\rightarrow K\rightarrow O_{X}\oplus L^{-1}\rightarrow O_{D}\rightarrow 0$$
$K$ sits inside the exact sequence:
$$0\rightarrow O_{X}(-D)\rightarrow K\rightarrow L^{-1}\rightarrow 0$$
We can recover the preimage of $(D,t_{D})$ by setting $M=L^{-1}(D)$, $E=K(D)$, and $E\rightarrow M\otimes L$ corresponds to
$$K\rightarrow O_{X}\oplus L^{-1}\rightarrow O_{X}$$
where the last arrow is the projection onto the first factor. Hence if $(D,t_{D})$ is a point on $\textrm{Hilb}^{d}_{C}$, then $O_{C}=O_{X}\oplus L^{-1}$, and we have a morphism of exact sequences:
$$\xymatrix{
o\ar[r] & K \ar[r] \ar[d]^{\phi} & O_{X}\oplus L^{-1} \ar[r] \ar[d] & O_{D} \ar[r] \ar[d] & 0\\
0\ar[r] & K\otimes L \ar[r] & L\oplus O_{X} \ar[r] & L_{D} \ar[r] & 0\\
}
$$
Hence $K\rightarrow O_{X}\oplus L^{-1}\rightarrow O_{X}$ can be identified with
$$K\xrightarrow{\phi}K\otimes L\rightarrow L\oplus O_{X}\rightarrow O_{X}$$
Hence the corresponding $E\rightarrow M\otimes L$ comes from $E\rightarrow E\otimes L$, this proves that $p^{-1}(D,t_{D})\in\hig'^{m}_{C}$.\\
For $(2)$, denote the intersection
$$\mathcal{Z}\times\hig\cap \pi(p^{-1}(\textrm{Hilb}^{d}_{\mathcal{C}|H}\times_{H}\hig))$$
by $\mathcal{W}$. Let $(z,F)$ be a point in $\mathcal{W}$, $b$ a point in $\mathcal{B}$ lying over $z$ such that $p(b,F)\in\textrm{Hilb}^{d}_{\mathcal{C}|H}\times_{H}\hig$. From Lemma~\ref{definition of Z} and the proof of part $(2)$ of Proposition~\ref{resolve the rational map}, we see that each point $z\in\mathcal{Z}$ determines a global section $t\in H^{0}(L)$. Moreover, the image of $b$ in $V$ consists of a finite subscheme $D$ of length $d$ of $X$, and $t$ induces a section in $H^{0}(L\otimes O_{D})$ which gives $D$ the structure of a closed subscheme of $S$. Hence $D$ lies inside the image of the section $X\xrightarrow{t} S$. Let $C$ be the spectral curve of $F$. Because $p(b,F)$ lies in $\textrm{Hilb}^{d}_{\mathcal{C}|H}\times_{H}\hig$, we see that $D$ is also a closed subscheme of $C$, hence $D\in C\cap X$ where $X$ is viewed as a curve in $S$ via the section $t$. If $X$ is not a component of $C$, then it is easy to see $X\cap C$ only has length $2l$, but $D$ has length $d=m+l>2l$ since we require that $m>l$ at the beginning of this subsection. Hence we must have that $X$ is a component of $C$. This again implies that the fibers of the projection $\mathcal{W}\rightarrow\mathcal{Z}$ has dimension at most $l+1-g+l+2(g-1)=2l+g-1$. Now the assertion follows from Lemma~\ref{definition of Z}\\

\end{proof}
For the purpose of the section $4$ and $5$, let us also note the following lemma:
\begin{lemma}\label{factorization of tau}
\begin{enumerate1}
\item There exists a smooth morphism $\mathcal{Y}\rightarrow \Pd$ and a regular embedding $$\mathcal{B}\hookrightarrow\mathcal{Y}\times_{\Pd}X^{(d)}$$.
\item Consider $(\hig'^{m}\backslash\mathcal{Z})\subseteq(\mathcal{Y}\backslash\mathcal{Z})$. We have the following commutative diagram
$$\xymatrix{
\hig'^{m}\backslash\mathcal{Z} \ar[r]^{p} \ar[d]^{u} & \textrm{Hilb}^{d}_{\mathcal{C}|H} \ar[d]^{v}\\
\hig \ar[r]^{\tau} & \hig
}
$$
where $u$ is the natural projection, $v$ sends $D$ to $\check{I}_{D}=\Hom(I_{D},O_{C})$ and $\tau$ is the involution of $\hig$ given by
$$E\rightarrow \check{E}\otimes\textrm{det}(E)$$
where the Higgs field on $\check{E}\otimes\textrm{det}(E)$ is induced from $E$.
\item Consider the following morphism
$$\hig\xrightarrow{l}\hig^{reg}\qquad (E,\phi)\rightarrow \lambda^{*}(\de(E))$$
where $\lambda$ is the projection from the spectral curve $C$ to $X$. Then the involution $\tau$ factors as:
$$\hig\xrightarrow{(\lambda,\textrm{id})}\hig^{reg}\times_{H}\hig\xrightarrow{\textrm{id}\times (\mu_{L}\circ\nu)}\hig^{reg}\times_{H}\hig\xrightarrow{\mu}\hig$$
where $\nu$, $\mu$ and $\mu_{L}$ are defined in Lemma~\ref{equivariance property}.
\end{enumerate1}
\end{lemma}
\begin{proof}
For part $(1)$, the morphism $\mathcal{Y}\rightarrow \Pd$ is given by:
$$(E,\varphi,s,\sigma)\rightarrow M\otimes L$$
It is easy to see that this morphism is smooth. Now consider the following diagram:
$$\xymatrix{
X\times\mathcal{Y}\ar[r] \ar[d]^{f} & X\times\Pd \ar[d]^{h}\\
\mathcal{Y} \ar[r]^{p} & \Pd\\
}
$$
Let $\mathcal{N}$ be the universal line bundle on $X\times\Pd$. $X^{(d)}$ is a projective bundle over $\Pd$ determined by the vector bundle $h_{*}(\mathcal{N})$, so $\mathcal{Y}\times_{\Pd}X^{(d)}$ is a projective bundle over $\mathcal{Y}$ corresponds to the vector bundle $f_{*}(M\otimes L)$. From Lemma~\ref{definition of Z}, $\mathcal{Z}$ is defined by the vanishing locus of a section of the vector bundle $f_{*}(M\otimes L)$, hence the assertion follows from Proposition~\ref{blowup as regular embedding}.\\
For part $(2)$, notice that part $(1)$ of Lemma~\ref{further props of B} implies that the restriction of $p$ to $\hig'^{m}\backslash\mathcal{Z}$ sends $\hig'^{m}\backslash\mathcal{Z}$ to $\textrm{Hilb}^{d}_{\mathcal{C}|H}$. Moreover, from the proof of part $(1)$ of Lemma~\ref{further props of B} we see that if
$$0\rightarrow K=I_{D}\rightarrow O_{X}\oplus L^{-1}\rightarrow O_{D}\rightarrow 0$$
is the image of $(E,\phi,s,\sigma)\in\hig'^{m}$, then we have
$$K\simeq E\otimes M^{-1}\otimes L^{-1}$$
It is easy to see that $\check{I}_{D}\simeq K^{-1}\otimes L^{-1}$, hence the assertion follows from this, using the fact that $\det(E)\simeq M$.\\
Part $3$ follows immediately from the definitions, using Lemma~\ref{description of the dual}.
\end{proof}

\section{A cohomological vanishing result}
In this section we keep the same notation and assumption as in subsection $3.3$. Set $d=m+l$ where the assumptions on $l$ and $m$ are given in the second paragraph of subsection $3.3$. The main result of this section is Lemma~\ref{the main lemma}. First let us define a certain open substack of $\hig^{n}$.
\begin{construction}\label{certain open substack}
Let $\hig^{(n)}$ denote the open substack of $\hig^{n}$ such that the underlying rank two vector bundle $E'$ satisfies the following conditions:
\begin{enumerate1}
\item All line subbudles of $E'$ has degree less than or equals to $-m-g+1$
\item All quotient line bundles of $E'$ has degree greater than or equals to $g-d=g-l-m$
\end{enumerate1}
\end{construction}
From the definition of semistability of vector bundles, it is easy to check that if we take $n=-2m-2g+2$ or $n=-2m-2g+3$, then all the semistable vector bundle of rank two of degree $n$ satisfies the two conditions above under our assumptions on $l$. Hence we can choose two consecutive integers $n$ for which $\hig^{(n)}$ is nonempty. For convenience, let us also give the following reformulation of the above two conditions in terms of the vanishing of certain cohomology groups:
\begin{lemma}
The two conditions in Construction~\ref{certain open substack} is equivalent to the following:
\begin{enumerate1}
\item $H^{1}(E'\otimes\Omega_{X}((d-g+1)x_{0})\otimes\mathcal{N})=0$ for all degree zero line bundle $\mathcal{N}$ on $X$
\item $H^{1}(\check{E'}\otimes L\otimes O_{X}(-(d-g)x_{0})\otimes\mathcal{N})=0$ for all degree zero line bundle $\mathcal{N}$ on $X$
\end{enumerate1}
\end{lemma}

The purpose of this section is to prove the following lemma:
\begin{lemma}\label{the main lemma}
Consider:
$$\xymatrix{
\mathcal{B}\times\hig^{(n)} \ar[r]^{p} \ar[d]^{\pi} & \HS^{d}\times\hig^{(n)}\\
\mathcal{Y}\times\hig^{(n)}
}
$$
Then we have:
\begin{enumerate1}
\item $R\pi_{*}(p^{*}(Q))\in D^{\leq 0}_{coh}(\mathcal{Y}\times\hig)$ where $p^{*}$ stands for the derived pullback functor.
\item $\RH(R\pi_{*}(p^{*}(Q)),O)\in D^{\leq d}_{coh}(\mathcal{Y}\times\hig)$.
\end{enumerate1}
Here let us remind the reader that $d=l+m$ is the same as the codimension of $\hig^{'m}\times_{H}\hig^{(n)}$ in $\mathcal{Y}\times\hig^{(n)}$.
\end{lemma}

Since $Q$ is a direct summand of $\psi_{*}(Q')$(Lemma~\ref{summand}), and the image of $\mathcal{B}$ is contained in the open subscheme $V$ of $\HS^{d}$($V$ is defined in Corollary~\ref{Hilbert scheme of P1 and S}) by part $(2)$ of Proposition~\ref{resolve the rational map}, we only need to prove the following:
\begin{lemma}\label{reduction of the main lemma}
Set $\textrm{Flag}^{d}_{\mathcal{B}}=\mathcal{B}\times_{V}V'$. Consider:
$$\xymatrix{
\textrm{Flag}^{d}_{\mathcal{B}}\times\hig^{(n)}\ar[r]^{w} \ar[d]^{\psi'} & V'\times\hig^{(n)} \ar[d]^{\psi}\\
\mathcal{B}\times\hig^{(n)} \ar[r]^{p} \ar[d]^{\pi} & V\times\hig^{(n)}\\
\mathcal{Y}\times\hig^{(n)}
}
$$
where $V$ is the open subscheme of $\HS^{d}$ defined in Corollary~\ref{Hilbert scheme of P1 and S} and $V'$ is the open subscheme of $\textrm{Flag}^{'d}_{S}$ defined in Corollary~\ref{Hilbert scheme of P1 and S}. Then we have
\begin{enumerate1}
\item $R\pi_{*}(\psi'_{*}(w^{*}(Q')))\in D^{\leq 0}_{coh}(\mathcal{Y}\times\hig)$. Here again all functors are derived.
\item $\RH(R\pi_{*}(\psi'_{*}(w^{*}(Q'))),O)\in D^{\leq d}_{coh}(\mathcal{Y}\times\hig)$.
\end{enumerate1}
\end{lemma}

The proof of the previous lemma relies on the following:
\begin{lemma}\label{existence of resolutions}
\begin{enumerate1}
\item There exists a Cartesian diagram:
$$\xymatrix{
\FB \ar[r] \ar[d] & X^{d} \ar[d]\\
\mathcal{B} \ar[r] & X^{(d)}\\
}
$$
and a closed embedding $\FB\hookrightarrow\mathcal{Y}\times_{\Pd} X^{d}$.
\item There exists a complex of locally free sheaves $K^{*}$ on $\textrm{Flag}^{d}_{\mathcal{B}}\times\hig^{(n)}$ concentrated in degree $[-d,0]$ representing $w^{*}(Q')$ where $w$ is the morphism:
$$\textrm{Flag}^{d}_{\mathcal{B}}\times\hig^{(n)}\xrightarrow{w}V'\times\hig^{(n)}$$
in Lemma~\ref{reduction of the main lemma}. Moreover, each term $K^{-p}$ is the pullback of a vector bundle $F^{-p}$ on $X^{d}\times\hig^{(n)}$ via
$$\FB\times\hig^{(n)}\hookrightarrow(\mathcal{Y}\times_{\Pd} X^{d})\times\hig^{(n)}\rightarrow X^{d}\times\hig^{(n)}$$
\end{enumerate1}
\end{lemma}

\begin{proof}
For $(1)$, notice that $\FB=\mathcal{B}\times_{V}V'$, so the claim follows from part $(2)$ of Corollary~\ref{Hilbert scheme of P1 and S} and part $(1)$ of Lemma~\ref{factorization of tau}\\
For part $(2)$, let $E'$ be the universal Higgs bundle on $X\times\hig^{(n)}$, viewed as a coherent sheaf on the surface $S$. Then by Lemma~\ref{resolution}, we see that:
$$(q^{*}(E')\otimes q^{*}(L^{-1})\rightarrow q^{*}(E'))^{\boxtimes d}$$
is a locally free resolution of $E^{'\boxtimes d}$ as coherent sheaves on $S^{d}\times\hig^{(n)}$. Now consider:
$$V'\times\hig^{(n)}\xrightarrow{\sigma} S^{d}\times\hig^{(n)}$$
It is proved in $[1]$ that $L\sigma^{*}(E^{'\boxtimes d})\simeq\sigma^{*}(E^{'\boxtimes d})$, hence
$$\sigma^{*}(q^{*}(E')\otimes q^{*}(L^{-1})\rightarrow q^{*}(E'))^{\boxtimes d}$$
is a locally free resolution of $\sigma^{*}(E^{'\boxtimes d})$ on $V'\times\hig^{(n)}$. By construction, $Q'$ is given by $\sigma^{*}(E^{' \boxtimes d})\otimes\psi^{*}(\de(O_{D'}))^{-1}$ where $D'$ is the universal subscheme of $S$ on $S\times V$. Now observe that by Corollary~\ref{Hilbert scheme of P1 and S}, the composition of
$$V'\xrightarrow{\sigma} S^{d}\xrightarrow{q} X^{d}$$
is identified with the natural projection
$$V'\xrightarrow{r'} X^{d}$$
Hence the composition of
$$\FB\times\hig^{(n)}\xrightarrow{w}V'\times\hig^{(n)}\xrightarrow{\sigma} S^{d}\times\hig^{(n)}\xrightarrow{q} X^{d}\times\hig^{(n)}$$
is identified with
$$\FB\times\hig^{(n)}\hookrightarrow(\mathcal{Y}\times_{\Pd} X^{d})\times\hig^{(n)}\rightarrow X^{d}\times\hig^{(n)}$$
Also, by Corollary~\ref{Hilbert scheme of P1 and S} and Lemma~\ref{properties of abel-jacobian map}, $\psi^{*}(\de(O_{D'})^{-1})\simeq r'^{*}(O(\sum_{i<j}\Delta_{ij}))$, hence the assertion follows from this.

\end{proof}
As a byproduct, we get the following more explicit description about the vector bundle $F^{-p}$ on $X^{d}\times\hig^{(n)}$:
\begin{corollary}\label{explicit description of resolutions}
Let $E'$ be the universal Higgs bunle on $X\times\hig^{(n)}$. Then the vector bundle $F^{-p}$ on $X^{d}\times\hig^{(n)}$  is a direct sum of vector bundles of the form
$$\mathcal{F}_{1}\boxtimes\cdots\boxtimes\mathcal{F}_{d}\otimes O(\sum_{i<j}\Delta_{ij})$$
where there exists a subset $I$ of $\{1,2,\cdots,d\}$ consisting of $p$ elements such that for all $i\in I$, $\mathcal{F}_{i}\simeq E'\otimes L^{-1}$, and $\mathcal{F}_{j}\simeq E'$ for all $j$ not in $I$.

\end{corollary}
We also need the following standard fact to prove Lemma~\ref{reduction of the main lemma}:
\begin{lemma}\label{spectral sequence}
Let $X\xrightarrow{f}Y$ be a proper morphism of schemes, and $M\in D^{b}_{coh}(X)$ represented by a complex of the form:
$$0\rightarrow C^{-n}\rightarrow\cdots\rightarrow C^{0}\rightarrow 0$$
If $R^{i}f_{*}C^{-j}=0$ for $i>j$, then $R^{p}f_{*}M=0$ for $p>0$.
\end{lemma}

Now we are ready to prove Lemma~\ref{reduction of the main lemma}. But before we enter into the proof, let us indicate the main ingredients of the proof in the following lemma:
\begin{lemma}\label{baby version of the proof}
Let $A$ be a scheme of finite type over a field $k$, $\mathcal{E}$ a vector bundle of rank $n$ on $A$ with a section $s$. Let $Z$ be the closed subscheme of $A$ given by the vanishing locus of $s$. Assume that $Z$ is a local complete intersection of codimension $n$ in $A$. Let $\blA$ be the blowup of $A$ along $Z$ and $\mathbf{P}$ be the projective bundle associated to $\mathcal{E}$ and $O(1)$ be the relative ample line bundle. Consider the diagram:
$$
\xymatrix{
\blA \ar[rr]^{i} \ar[rd]_{\pi} & & \mathbf{P} \ar[ld]^{\pi'}\\
 & A &\\
}
$$
Assume that $K\in D^{b}_{coh}(\blA)$ is represented by a complex of the form
$$0\rightarrow C^{-r}\rightarrow C^{-(r-1)}\rightarrow\cdots\rightarrow C^{0}\rightarrow 0$$
such that each term $C^{-i}$ is the pullback of a vector bundle $D^{-i}$ on $\mathbf{P}$. If $R^{j}\pi'_{*}(D^{-i}\otimes O(a))=0$ for all $j>i$ and $a\geq 0$, then we have $R\pi_{*}(K)\in D^{\leq 0}_{coh}(A)$.
\end{lemma}
\begin{proof}
By Lemma~\ref{pushforward of bl}, the assumptions on $D^{-i}$ implies that $R^{j}\pi_{*}(C^{-i})=0$ for $j>i$, hence the claim follows from Lemma~\ref{spectral sequence}.

\end{proof}

Now let us start the proof of Lemma~\ref{reduction of the main lemma}:
\begin{proof}
Let us deal with part $(1)$ first. We want to apply Lemma~\ref{baby version of the proof} with $A=\mathcal{Y}\times\hig^{(n)}$, $\blA=\mathcal{B}\times\hig^{(n)}$. The argument is divided into the following steps:\\
\\
Step $1$: $\psi'_{*}(w^{*}(Q'))\in D^{b}_{coh}(\mathcal{B}\times\hig^{(n)})$ is represented by a complex of locally free sheaves of the form:
$$0\rightarrow C^{-d}\rightarrow C^{-(d-1)}\rightarrow\cdots\rightarrow C^{0}\rightarrow 0$$
Indeed, from Lemma~\ref{existence of resolutions}, we see that $w^{*}(Q')$ on $\textrm{Flag}^{d}_{\mathcal{B}}\times\hig^{(n)}$ is represented by a complex of locally free sheaves of the form
$$0\rightarrow K^{-d}\rightarrow K^{-(d-1)}\rightarrow\cdots\rightarrow K^{0}\rightarrow 0$$
Since $\psi'$ is finite flat of degree $d!$, we are done.\\
\\
Step $2$: Set $\mathbf{P}=\mathcal{Y}\times_{\Pd}X^{(d)}$. Then $\mathbf{P}$ is a projective bundle over $\mathcal{Y}$ associated with a certain vector bundle $\mathcal{E}$ on $\mathcal{Y}$. Moreover, $\mathcal{Z}$ is the vanishing locus of a section of $\mathcal{E}$. Hence we have the diagram:
$$\xymatrix{
\mathcal{B} \ar[rr]^{i} \ar[rd]_{\pi} & & \mathbf{P} \ar[ld]^{f'}\\
 & \mathcal{Y} &\\
}
$$
Indeed, the claim follows directly from part $(1)$ of Lemma~\ref{factorization of tau}\\
\\
Step $3$: Each term of the complex $C^{-i}$ on $\mathcal{B}\times\hig^{(n)}$ in Step $1$ is the pullback of a vector bundle $D^{-i}$ on $\mathbf{P}\times\hig^{(n)}$.\\
Indeed, set $\mathbf{P}'=\mathcal{Y}\times_{\Pd}X^{d}$. By part $(1)$ of Lemma~\ref{factorization of tau}, part $(2)$ of Corollary~\ref{Hilbert scheme of P1 and S} and the definition of $\FB$, we have the following commutative diagram with Cartesian squares:
$$\xymatrix{
\FB\times\hig^{(n)} \ar[r]^{i'} \ar[d]^{\psi'} & \mathbf{P}'\times\hig^{(n)} \ar[r]^{\theta'} \ar[d] & X^{d}\times\hig^{(n)} \ar[d]^{h}\\
\mathcal{B}\times\hig^{(n)} \ar[r]^{i} & \mathbf{P}\times\hig^{(n)} \ar[r]^{\theta} & X^{(d)}\times\hig^{(n)}\\
}
$$
By Lemma~\ref{existence of resolutions}, each $K^{-i}$ on $\FB\times\hig^{(n)}$ is the pullback of a vector bundle $F^{-i}$ on $X^{d}\times\hig^{(n)}$, hence each $C^{-i}=\psi'_{*}(K^{-p})$ is the pullback of $D^{-i}=\theta^{*}(h_{*}(F^{-i}))$ on $\mathbf{P}\times\hig^{(n)}$.\\
\\
Step $4$: Consider
$$\xymatrix{
\mathbf{P}\times\hig^{(n)} \ar[d]_{f'}\\
\mathcal{Y}\times\hig^{(n)}\\
}
$$
Then $R^{j}f'_{*}(D^{-i}\otimes O(a))=0$ for all $j>i$ and $a\geq 0$.\\
To see this, consider the following diagram where the bottom square is Cartesian:
$$\xymatrix{
 & X^{d}\times\hig^{(n)} \ar[d]^{h}\\
\mathbf{P}\times\hig^{(n)} \ar[r]^{\theta} \ar[d]^{f'} & X^{(d)}\times\hig^{(n)} \ar[d]^{f}\\
\mathcal{Y}\times\hig^{(n)} \ar[r] & \Pd\times\hig^{(n)}\\
}
$$
By Step $3$, each $D^{-i}=\theta^{*}(h_{*}(F^{-i}))$, so from part $1$ of Lemma~\ref{factorization of tau} we see that $\mathcal{Y}\rightarrow\Pd$ is smooth, hence we only need to prove:
$$R^{j}f_{*}(h_{*}(F^{-i})\otimes O(a))=0$$
for $j>i$ and $a\geq 0$. Moreover, by projection formula, it is enough to prove that
$$R^{j}f_{*}(F^{-i}\otimes h^{*}(O(a))\otimes f^{*}(O(\Theta)))=0$$
for $j>i$ and $a\geq 0$ where $\Theta$ is the theta divisor on $\Pd$. Now the claim is an easy consequence of Lemma~\ref{vanishing of pushforward}, using the description of $F^{-i}$ in Lemma~\ref{explicit description of resolutions}, the description of $O(\Theta)$ and $O(1)$ in Lemma~\ref{properties of abel-jacobian map} and our assumptions on $E'$ in the beginning of this subsection.\\
Now from Lemma~\ref{baby version of the proof}, the proof of part $1$ of Lemma~\ref{reduction of the main lemma} is complete.\\
\\
Now let us turn to the proof of part $2$ of Lemma~\ref{reduction of the main lemma}. The proof of similar to the argument in part $1$. First notice that by Grothendieck duality and Proposition~\ref{dualizing sheaf of bl}, we only need to prove
$$R\pi_{*}\RH(\psi'_{*}(w^{*}(Q')),O((d-g)E))\in D^{\leq d}_{coh}(\mathcal{Y}\times\hig)$$
where $E$ is the exceptional divisor of the blowup $\mathcal{B}\rightarrow\mathcal{Y}$.\\
\\
Step $1$: $\RH(\psi'_{*}(w^{*}(Q')),O((d-g)E))\in D^{\leq d}_{coh}(\mathcal{Y}\times\hig)$ is represented by a complex of the form:
$$0\rightarrow C^{0}\rightarrow C^{1}\rightarrow\cdots\rightarrow C^{d}\rightarrow 0$$
Indeed, consider the following diagram with Cartesian squares:
$$\xymatrix{
\FB\times\hig^{(n)} \ar[r]^{i'} \ar[d]^{\psi'} & \mathbf{P}'\times\hig^{(n)} \ar[r]^{\theta'} \ar[d] & X^{d}\times\hig^{(n)} \ar[d]^{h}\\
\mathcal{B}\times\hig^{(n)} \ar[r]^{i} & \mathbf{P}\times\hig^{(n)} \ar[r]^{\theta} & X^{(d)}\times\hig^{(n)}\\
}
$$
Because of part $(1)$ of Lemma~\ref{existence of resolutions}, the relative dualizing sheaf $\omega$ of the morphism $\psi'$ is isomorphic to $i'^{*}\theta'^{*}(\omega_{X^{d}|X^{(d)}})$. Since $w^{*}(Q')$ is represented by the complex of locally free sheaves $K^{*}$, we conclude from duality that
$$\RH(\psi'_{*}(w^{*}(Q')),O((d-g)E))$$
is represented by a complex of the prescribed form with
$$C^{i}=\psi'_{*}(\check{K}^{-i}\otimes\omega)\otimes O((d-g)E)$$.\\
\\
Step $2$: Each term $C^{i}$ is the pullback of a vector bundle $D^{i}$ on $\mathbf{P}\times\hig^{(n)}$.
Indeed, each $\check{K}^{-i}$ is the pullback of $\check{F}^{-i}$ on $X^{d}\times\hig$, and $\omega$ is also the pullback of $\omega_{X^{d}|X^{(d)}}$ on $X^{d}\times\hig$, and $O(E)$ is the pullback of $O(-1)$ from $\mathbf{P}\times\hig^{(n)}$, hence each $C^{i}$ is the pullback of
$$\theta^{*}(h_{*}(\check{F}^{-i}\otimes\omega_{X^{d}|X^{(d)}}))\otimes O(-1)^{\otimes (d-g)}$$
on $\mathbf{P}\times\hig$.\\
\\
Step $3$: We have $R^{j}f'_{*}(D^{i}\otimes O(a))=0$ for all $j>d-i$ and $a\geq 0$.
To see this, use diagram:
$$\xymatrix{
 & X^{d}\times\hig^{(n)} \ar[d]^{h}\\
\mathbf{P}\times\hig^{(n)} \ar[r]^{\theta} \ar[d]^{f'} & X^{(d)}\times\hig^{(n)} \ar[d]^{f}\\
\mathcal{Y}\times\hig^{(n)} \ar[r] & \Pd\times\hig^{(n)}\\
}
$$
where the bottom square is Cartesian. By step $2$, each $D^{i}\otimes O(a)$ is the pullback of
$$h_{*}(\check{F}^{-i}\otimes\omega_{X^{d}|X^{(d)}})\otimes O(a-d+g)$$
from $X^{(d)}\times\hig^{(n)}$. Part $1$ of Lemma~\ref{factorization of tau} implies that $\mathcal{Y}\rightarrow\Pd$ is smooth, hence we only need to prove:
$$R^{j}f_{*}(h_{*}(\check{F}^{-i}\otimes\omega_{X^{d}|X^{(d)}})\otimes O(a-d+g))=0$$
for all $i>d-i$ and $a\geq 0$. This again follows from Lemma~\ref{properties of abel-jacobian map} and our assumptions on $E'$.\\
From Lemma~\ref{baby version of the proof} again, the proof of part $2$ of Lemma~\ref{reduction of the main lemma} is also complete.

\end{proof}

\section{The proof of the main theorem}
In this last section we will prove the main theorem. First let us establish the following:
\begin{lemma}
Consider the diagram in Lemma~\ref{the main lemma}
$$\xymatrix{
\mathcal{B}\times\hig^{(n)} \ar[r]^{p} \ar[d]^{\pi} & V\times\hig^{(n)}\\
\mathcal{Y}\times\hig^{(n)}
}
$$
Then
\begin{enumerate1}
\item $R\pi_{*}(p^{*}(Q))$ is a maximal Cohen-Macaulay sheaf of codimension $d$ on $\mathcal{Y}\times\hig^{(n)}$ supported in $\hig'^{m}\times_{H}\hig^{(n)}$. Let us remind the reader again that $d$ is equal to the codimension of $\hig'^{m}\times_{H}\hig^{(n)}$ inside $\mathcal{Y}\times\hig^{(n)}$.
\item Let $U_{m}$ be the open substack of $\hig^{m}$ given by the condition that the underlying vector bundle $E$ of the Higgs bundle satisfies the condition $H^{1}(E)=0$ and $H^{0}(\check{E}\otimes L)=0$. Then $R\pi_{*}(p^{*}(Q))$ descends down to a maximal Cohen-Macaulay sheaf $\mathcal{P}'$ on $U_{m}\times_{H}\hig^{(n)}$ such that its restriction to $(\hig^{reg}\cap U_{m})\times_{H}\hig^{(n)}$ agrees with the pullback of the Poincar\'e line bundle along
$$\hig^{reg}\times_{H}\hig^{(n)}\xrightarrow{\tau\times\textrm{id}}\hig^{reg}\times_{H}\hig^{(n)}$$
where $\tau$ is the involution of $\hig$ defined in part $(2)$ of Lemma~\ref{factorization of tau}.
\end{enumerate1}
\end{lemma}
\begin{proof}
We are going to prove this by applying Lemma $7.7$ of $[1]$. By Lemma~\ref{the main lemma} and Lemma $7.7$ of $[1]$, we only need to check that the support of $R\pi_{*}(p^{*}(Q))$ has codimension greater than or equals to $d$. Let $\mathcal{W}$ denote the support of $R\pi_{*}(p^{*}(Q))$. First we claim that if we restrict $p$ to $(\mathcal{Y}\backslash\mathcal{Z})\times\hig$, then we have
$$p^{-1}(\textrm{Hilb}^{d}_{\mathcal{C}|H}\times_{H}\hig)=(\hig'^{m}\cap(\mathcal{Y}\backslash\mathcal{Z}))\times_{H}\hig$$
In fact, since the restriction of $p$ to $\mathcal{Y}\backslash\mathcal{Z}$ is smooth by part $(3)$ of Proposition~\ref{resolve the rational map}, and both $\textrm{Hilb}^{d}_{\mathcal{C}|H}\times_{H}\hig$ and $\hig'^{m}\times_{H}\hig$ are integral, we only need to check the equality set theoretically, but this is the content of part $(1)$ of Lemma~\ref{further props of B}. From this we conclude that
$$\mathcal{W}\cap((\mathcal{Y}\backslash\mathcal{Z})\times\hig)=(\hig'^{m}\cap(\mathcal{Y}\backslash\mathcal{Z}))\times_{H}\hig$$
Moreover, from its definition, it is clear that $\mathcal{W}\subseteq\pi(p^{-1}(\textrm{Hilb}^{d}_{\mathcal{C}|H}\times_{H}\hig))$, hence by part $(2)$ of Lemma~\ref{further props of B} we have:
$$\textrm{dim}(\mathcal{W}\cap(\mathcal{Z}\times\hig))\leq m+3l+2g-1$$
On the other hand, part $(5)$ of Lemma~\ref{elementary properties of Y} implies that
$$\textrm{dim}(\hig'^{m}\times_{H}\hig)=m+5l+1$$
Hence we have:
$$\textrm{dim}(\hig'^{m}\times_{H}\hig)>\textrm{dim}(\mathcal{W}\cap(\mathcal{Z}\times\hig))$$
Combine these with part $(3)$ of Lemma~\ref{elementary properties of Y} we have:
$$\textrm{codim}(\mathcal{W})\geq d$$
Hence $R\pi_{*}(p^{*}(Q))$ is a Cohen-Macaulay sheaf of codimension $d$ on $\mathcal{Y}\times\hig^{(n)}$. Moreover, because
$$\textrm{dim}(\hig'^{m}\times_{H}\hig)=m+5l+1>\textrm{dim}(\mathcal{W}\cap(\mathcal{Z}\times\hig))$$
the Cohen-Macaulayness implies that
$$\mathcal{W}\cap(\mathcal{Y}\backslash\mathcal{Z}\times\hig^{(n)})\subseteq\hig'^{m}\times_{H}\hig^{(n)}$$
is dense in $\mathcal{W}$, hence $\mathcal{W}=\hig'^{m}\times_{H}\hig^{(n)}$.\\
For $(2)$, notice that by our construction and the discussions in Proposition~\ref{definition of Q} and part $(2)$ of Lemma~\ref{factorization of tau}, $R\pi_{*}(p^{*}(Q))$ agrees with the pullback of $(\tau\times \textrm{id})^{*}(\mathcal{P})$ on the open substack:
$$(\hig'^{m}\cap(\mathcal{Y}\backslash\mathcal{Z}))\times_{H}\hig^{(n)}$$
where $\mathcal{P}$ is the Poincar\'e sheaf on $widetilde{\hig}\times_{H}\hig^{(n)}$ and $\tau$ is the involution of $\hig$ in part $(2)$ of Lemma~\ref{factorization of tau}.
From the proof of $(1)$ we see that the codimension of
$$(\hig'^{m}\cap\mathcal{Z})\times_{H}\hig^{(n)}$$
in
$$\hig'^{m}\times_{H}\hig^{(n)}$$
is greater than or equals to $3$. So by Proposition~\ref{uniqueness of CM sheaf} the descend data extends to $\hig'^{m}\times_{H}\hig^{(n)}$. So by part $(4)$ of Lemma~\ref{elementary properties of Y} and part $(2)$ of Lemma~\ref{factorization of tau}, it descends to $U_{m}\times_{H}\hig^{(n)}$ and agrees with the pullback of the Poincar\'e line bundle along $\tau\times\textrm{id}$.

\end{proof}

With the previous lemma at hand, we can now prove the main theorem. First notice that from the previous lemma we have a maximal Cohen-Macaulay sheaf $\mathcal{P}'$ on $U_{m}\times_{H}\hig^{(n)}$. Using Lemma~\ref{factorization of tau} and Lemma~\ref{equivariance property}, we see that there exists a line bundle $\mathcal{A}$ on $\hig\times_{H}\hig$ such that the restriction of
$$\check{\mathcal{P}}'\otimes\mathcal{A}$$
to
$$(U_{m}\cap\widetilde{\hig})\times_{H}\hig^{(n)}$$
agrees with the Poincar\'e sheaf on
$$(U_{m}\cap\widetilde{\hig})\times_{H}\hig^{(n)}$$
constructed in section $1.3$. Denote $\check{\mathcal{P}}'\otimes\mathcal{A}$ by $\mathcal{P}_{m}$. So our goal is to extend $\mathcal{P}_{m}$ to $\hig\times_{H}\hig_{ss}$ such that it agrees with the Poincar\'e line bundle on $\hig^{reg}\times_{H}\hig_{ss}$. First choose two consecutive integers $n_{1}$ and $n_{2}$  such that the corresponding stack $\hig^{(n_{i})}$ is nonempty. In fact, by our discussions in Construction~\ref{certain open substack} in section $4$, we can take
$$n_{1}=-2m-2g+2\qquad n_{2}=-2m-2g+3$$
The proof of the main theorem boils down to the following three claims:
\begin{Claim}
In order to construct the Poincar\'e sheaf on $\hig\times_{H}\hig_{ss}$, it is enough to construct the Poincar\'e sheaf on $\hig\times_{H}\hig^{n_{i}}_{ss}$ for all $i$.
\end{Claim}

\begin{Claim}
To construct the Poincar\'e sheaf on $\hig\times_{H}\hig^{n_{i}}_{ss}$, it is enough to construct the Poincar\'e sheaf on $U_{m}\times_{H}\hig^{n_{i}}_{ss}$ for all $m>>0$.
\end{Claim}

\begin{Claim}
Let $\hig^{(n_{i})}_{ss}=\hig^{(n_{i})}\cap\hig^{n_{i}}_{ss}$. To construct the Poincar\'e sheaf on $U_{m}\times_{H}\hig^{n_{i}}_{ss}$, it is enough to construct the Poincar\'e sheaf on $U_{m}\times_{H}\hig^{(n_{i})}_{ss}$.
\end{Claim}

\begin{proof}
For claim $(1)$ notice that by tensoring with the line bundle $O_{X}(x_{0})$, we can translate any $\hig^{l}$ into $\hig^{n_{i}}$ for some $i$. Hence from Lemma~\ref{equivariance property} and Proposition~\ref{uniqueness of CM sheaf}, if we can construct the Poincar\'e sheaf on $\hig\times_{H}\hig^{n_{i}}_{ss}$, then we can extend it to the entire $\hig\times_{H}\hig_{ss}$ using Lemma~\ref{equivariance property}, and its restriction to $\hig^{reg}\times_{H}\hig_{ss}$ will agree with the Poincar\'e line bundle.\\

For $(2)$, notice that for any rank $2$ Higgs bundle $E$, we can find an integer $n$ such that $E\otimes O_{X}(nx_{0})\in U_{m}$ for some $m$. So if we fix an integer $l$ and consider the isomorphism
$$\hig^{l}\xrightarrow{\alpha}\hig^{l+2n}$$
induced by tensoring with $O_{X}(nx_{0})$, then the union of $\alpha^{-1}(U_{l+2n})$ for all $n$ will cover $\hig^{l}$. Hence it suffices to construct the Poincar\'e sheaf on $\alpha^{-1}(U_{l+2n})\times_{H}\hig^{n_{i}}_{ss}$ for all $n>>0$. Now if we have the Poincar\'e sheaf on $U_{l+2n}\times_{H}\hig^{n_{i}}_{ss}$, then we can again use Lemma~\ref{equivariance property} to construct the Poincar\'e sheaf on $\alpha^{-1}(U_{l+2n})\times_{H}\hig^{n_{i}}_{ss}$, and Proposition~\ref{uniqueness of CM sheaf} will guarantee that it agrees with the Poincar\'e line bundle.\\

For $(3)$, by Lemma~\ref{a property of semistable higgs bundle}, we have that
$$\hig^{n}_{ss}=(\hig^{n}_{ss}\cap\widetilde{\hig})\cup(\hig^{(n)}_{ss})$$
The construction in $(1)$ already gives the Poincar\'e sheaf on $$U_{m}\times_{H}(\hig^{n}_{ss}\cap\widetilde{\hig})$$
Since the complement of $\hig^{reg}$ has codimension greater than or equals to two, it follows that the Poinca\'e sheaf is uniquely determined by its restriction to $\hig^{reg}\times_{H}\hig$ by Proposition~\ref{uniqueness of CM sheaf}. Hence if we can construct the Poincar\'e sheaf on $U_{m}\times_{H}\hig^{(n_{i})}_{ss}$ such that its restriction to $\hig^{reg}\times_{H}\hig^{(n_{i})}_{ss}$ agrees with the Poincar\'e line bundle, then it automatically compatible with the Poincar\'e sheaf on $U_{m}\times_{H}(\hig^{n_{i}}_{ss}\cap\widetilde{\hig})$, hence they glue.\\
For the claim about the flatness over $\hig_{ss}$, we use Proposition~\ref{flatness of CM sheaves} and Proposition~\ref{smoothness of semistable higgs bundles}.

\end{proof}


\begin{thebibliography}{10}

\bibitem{Autoduality}
D. Arinkin.
\emph{Autoduality of compactified Jacobian of planar curves}.
\newblock {arxiv:1001.3868v2}.

\bibitem{Poincare line bundle}
D. Arinkin.
\emph{Cohomology of line bundles on compactified Jacobians}.
\newblock {arxiv:0705.0190}.

\bibitem{Fourier-Mukai}
Margarida Melo, Antonio Rapagnetta, Filippo Viviani.
\emph{Fourier-Mukai and autoduality for compactified Jacobians II}.
\newblock {arxiv:1308.0564}.

\bibitem{Oper}
D. Arinkin.
\emph{Irreducible connections admit generic oper structures}.
\newblock {arxiv:1602.08989}.

\bibitem{commutative algebra}
David Eisenbud.
\emph{Commutative algebra with a view toward algebraic geometry}.

\bibitem{Stacks Project}
\emph{Stacks Project}.

\bibitem{irreducibility}
A.~Altman, A.~Iarrobino, and S.~Kleiman.
\emph{Irreducibility of the compactified Jacobian}.
\newblock {In {\em Real and complex singularities ({P}roc. {N}inth {N}ordic
  {S}ummer {S}chool/{NAVF} {S}ympos. {M}ath., {O}slo, 1976)}, pages 1--12.
  Sijthoff and Noordhoff, Alphen aan den Rijn, 1977}.

\bibitem{CJ}
A.~Altman and S.~Kleiman.
\emph{Compactifying the Picard scheme}.
\newblock {Adv. in Math., 35(1):50–112, 1980}.

\bibitem{Surface}
J. Fogarty.
\emph{Algebraic families on an algebraic surface}.
\newblock {Amer. J. Math, 90:511--521, 1968}.

\bibitem{Fourier-Mukai}
Daniel Huybrechts.
\emph{Fourier-Mukai Transforms in Algebraic Geometry}.
\newblock {Clarendon Press, 2006}.

\bibitem{Automorphic}
V. Drinfeld.
\emph{Two dimensional $l$-adic representations of the fundamental group of a curve over a finite field and automorphic forms on GL(2)}.
\newblock {Amer. J. Math. 105 (1983) 85-114}.

\bibitem{Laumon}
G. Laumon.
\emph{Correspondence de Langlands geometrique pour les corps de fonctions}.
\newblock {Duke Math. Jour. 54 (1987), 309–359}.

\bibitem{Mao Li}
M. Li.
\emph{Construction of the Poincar\'e sheaf on the stack of rank two Higgs bundles of $\p$}.
\newblock {arXiv:1709.05292}.

\bibitem{Abelian variety}
D. Mumford.
\emph{Abelian varieties},
\newblock {Tata Institute of Fundamental Research Studies in Mathematics, No. 5. Published for the Tata Institute of Fundamental Research, Bombay, 1970.}

\bibitem{Autoduality of compactified jacobian}
E. Esteves, M. Gagne and S. Kleiman.
\emph{Autoduality of the compactified Jacobian}.
\newblock {arXiv:math/9911071}.

\bibitem{Spectral curves}
A. Beauville, M.S. Narasimhan and S. Ramanan.
\emph{Spectral curves and the generalised theta divisor}.
\newblock {J. Reine Angew. Math. 398 (1989), 169–179.}

\bibitem{isospectral Hilbert scheme}
M. Haiman.
\emph{Hilbert schemes, polygraphs and the Macdonald positivity conjecture}.
\newblock {J. Amer.Math. Soc., 14(4):941–1006, 2001.}

\bibitem{Cohen-Macaulay}
W. Bruns and J. Herzog.
\emph{Cohen-Macaulay rings},
\newblock {Cambridge Studies in Advanced Mathematics. 39. Cambridge: Cambridge University Press, 1998.}

\end{thebibliography}
\end{document}